\documentclass[11pt]{article}

\usepackage[utf8]{inputenc}
\usepackage[T1]{fontenc}
\usepackage[english]{babel}
\usepackage{csquotes}
\usepackage{lipsum}
\usepackage{amsmath, amssymb, amsthm, mathtools}
\usepackage{dsfont}        

\usepackage{graphicx}
\usepackage{geometry}
\usepackage{caption}
\usepackage{subcaption}
\usepackage{float}
\usepackage{multirow}
\usepackage{array}
\usepackage{algorithmic}
\usepackage{booktabs}
\usepackage{authblk}

\usepackage{hyperref}
\usepackage[numbers,sort&compress]{natbib}
\usepackage{nameref}

\usepackage{icomma}        
\usepackage[dvipsnames]{xcolor} 
\usepackage{comment}       
\usepackage{todonotes}     
\usepackage{algorithm}
\usepackage{algorithmic}
\usepackage{draftwatermark}
\usepackage{colortbl}
\usepackage{ulem}  

\definecolor{a}{rgb}{0.83, 0.83, 0.83}
\definecolor{b}{rgb}{0.99,0.99,0.99}
\SetWatermarkText{}
\SetWatermarkScale{1}
\SetWatermarkColor[gray]{0.9}

\newcommand{\et}{\textit{et al.}}
\newcommand{\bl}[1]{\textcolor{black}{#1}}

\theoremstyle{plain}
\newtheorem{theorem}{Theorem}[section]
\newtheorem{proposition}{Proposition}[section]
\newtheorem{lemma}{Lemma}[section]
\newtheorem{corollary}{Corollary}[section]

\theoremstyle{definition}
\newtheorem{definition}{Definition}[section]

\newtheorem{remark}{Remark}[section]

\begin{document}


\pagenumbering{arabic}
\title{Theoretical Analysis of Thermodynamic Matrix Inversion: First-order Equivalence to 
Preconditioned Gradient Descent and Implications for Analog Computing}
\author{Gerhard Kirsten\textsuperscript{1,}\textsuperscript{*}, 
Michael Selby\textsuperscript{1}, 
Janith Petangoda\textsuperscript{1},\\
Olle Halqvist Elias\textsuperscript{1}, James Meech\textsuperscript{1}, 
Phillip Stanley-Marbell\textsuperscript{1,2}}
\date{}
\maketitle
\footnotetext[1]{Signaloid, United Kingdom}
\footnotetext[2]{University of Cambridge, United Kingdom}
\renewcommand{\thefootnote}{*}
\footnotetext{Corresponding author: gerhard.kirsten@signaloid.com}
\thispagestyle{empty}

\begin{abstract}
%
%
{
Recent research has demonstrated the possibility of exploiting the 
thermodynamics of coupled electrical oscillators to implement computational 
tasks such as matrix inversion. While physical implementations rely on thermal 
noise to drive equilibration, we show that the underlying dynamics reduce to 
a deterministic iterative algorithm. Building on the framework of Aifer \et, 
we analyze the moment evolution of the Ornstein-Uhlenbeck process governing 
thermodynamic symmetric positive definite (SPD) matrix inversion. We prove that to a first-order approximation, the covariance dynamics are 
mathematically identical to preconditioned gradient descent on the Frobenius 
norm of the residual $\tilde{A}^{-1}A-I$. This equivalence demonstrates that 
thermal fluctuations, while essential for physical thermodynamic hardware, 
are algorithmically redundant for convex problems with a single global minimum. 
We validate the resulting algorithm against Thermox (Duffield \et), a 
stochastic thermodynamic simulator, achieving speedups exceeding 100,000-fold 
while remaining competitive with the Newton-Schulz iteration. We also 
demonstrate acceleration through Schur complement techniques. These results 
establish a rigorous link between analog thermodynamic computing, statistical 
physics, and deterministic optimization methods.
}

\end{abstract}

\section*{Introduction}
Matrix inversion represents one of the most fundamental yet computationally-demanding operations 
in modern scientific computing, serving as a cornerstone for algorithms across diverse fields including 
Bayesian inference~\cite{leonard_bayesian_1992}, computational physics~\cite{franklin_matrix_2013}, 
medical imaging~\cite{cahill_preliminary_1970, yeung_pinv_2024}, portfolio optimization~\cite{markowitz_portfolio_1952}, 
numerical weather prediction~\cite{daley_atmospheric_1991, bannister_review_2017}, and machine learning applications. 
\bl{The ubiquity of matrix inversion stems from both its theoretical significance and practical utility across scales, from 
inverting small dense matrices in statistical parameter estimation to handling large sparse systems in climate modeling, 
where operational weather centers routinely process millions of observations~\cite{bannister_review_2017}}.

While linear systems are typically solved directly rather than through matrix inversion, particularly when coefficient 
matrices are sparse or structured~\cite{saad_iterative_2003}, computing the inverse becomes essential in several critical 
scenarios. Matrix inversion is preferred when solving multiple systems sharing the same coefficient matrix, enabling 
efficient reuse rather than repeated factorizations~\cite{golub_matrix_2013, trefethen_numerical_1997}. Moreover, matrix 
inversion underpins closed-form solutions in statistics and data analysis, including linear regression, covariance 
estimation~\cite{press_numerical_2007, watkins_fundamentals_2010}, multivariate distance computations such as the Mahalanobis 
distance~\cite{mahalanobis_generalised_1936}, and mean-variance portfolio optimization~\cite{markowitz_portfolio_1952}. In 
statistical applications, matrix conditioning becomes increasingly important as the number of \bl{variables, $p$, grows relative 
to the number of observations, $n$~\cite{ledoit_well-conditioned_2004, tabeart_improving_2019}}. Additionally, approximate inverses 
serve as critical preconditioners in iterative methods for large-scale problems~\cite{demmel_applied_1997, higham_accuracy_2002}, 
while specialized applications such as variational data assimilation require robust inversion of observation error covariance 
matrices~\cite{tabeart_improving_2019}.

Despite its central importance, matrix inversion faces computational and numerical stability limitations that constrain the 
scalability of modern algorithms. Direct inversion methods, such as \bl{Gaussian-elimination-based} approaches~\cite{golub_matrix_2013}, 
require $O(n^3)$ floating point operations and $O(n^2)$ storage while potentially becoming numerically unstable for \bl{poorly-conditioned} 
matrices. The conditioning challenge appears across many applications: in high-dimensional statistics where sample covariance matrices 
can become singular when $p > n$~\cite{ledoit_well-conditioned_2004}, in discretized partial differential equations where condition numbers 
scale with mesh refinement, and in data assimilation where interchannel correlations can lead to challenging matrix properties~\cite{tabeart_improving_2019}. 
To address these challenges, a rich literature has developed on approximate matrix inversion, including iterative methods such as the 
Schulz--Hotelling algorithm \cite{schulz_iterative_1933, hotelling_analysis_1943} and its variants~\cite{soleymani_2012, soleymani_2013}, 
which achieve high-order convergence while exploiting sparsity and parallelism for scalable computation.

The application of randomized linear algebra techniques to approximate matrix inversion has emerged as a particularly powerful paradigm for 
large-scale problems. \bl{Gower and Richtárik~\cite{gower_randomized_2017} developed a stochastic framework that reinterprets classical quasi-Newton methods, 
such as the Broyden-Fletcher-Goldfarb-Shanno (BFGS) and Davidon-Fletcher-Powell (DFP) algorithms. This framework treats these methods and the Broyden family as 
randomized matrix inversion algorithms with global linear convergence guarantees}. Building on this foundation, Richtárik and Takáč~\cite{richtarik_stochastic_2020} 
generalized these randomized approaches to the 
broader problem of solving consistent linear systems. Complementing these algorithmic advances, Wenger and Hennig~\cite{wenger_probabilistic_2020} 
introduced a probabilistic perspective that explicitly quantifies numerical uncertainty arising from finite computational resources, particularly 
valuable for structured matrices like kernel Gram matrices prevalent in machine learning applications. Among these emerging approaches, thermodynamic 
computing represents perhaps the most novel departure from conventional computational paradigms. Aifer \et~\cite{aifer_thermodynamic_2024} 
demonstrated that classical thermodynamic principles can be harnessed for linear algebraic computations, including SPD matrix inversion, by mapping 
linear systems onto coupled harmonic oscillator networks that exploit thermal equilibration for analog computation. Their approach leverages the 
natural dynamics of physical systems rather than discrete arithmetic operations, treating thermal noise as a computational resource rather than a source of error.

{Thermodynamic computing represents a distinct paradigm within the broader landscape of physics-based analog computing. It sits alongside approaches 
such as oscillatory neural networks, which exploit the collective dynamics of coupled oscillators for various computational tasks. These dynamics range from 
synchronization to chaotic behavior and are applied across pattern recognition, combinatorial optimization, and signal processing~\cite{todri_sanial_computing_2024}}. 
{However, unlike oscillatory networks that often encode information through phase relationships, thermodynamic computing fundamentally 
relies on stochastic thermal fluctuations to drive the system toward statistical equilibrium.} While most oscillatory computing research has focused on 
applications such as associative memory and graph coloring problems, the connection between physical dynamics and fundamental linear algebra operations 
remained largely unexplored until recently~\cite{aifer_thermodynamic_2024}.

This work provides a theoretical analysis of thermodynamic approaches to SPD matrix inversion, revealing that the stochastic processes underlying these methods 
possess an elegant mathematical structure. {In the framework of Aifer \et~\cite{aifer_thermodynamic_2024}, the matrix to be inverted encodes 
the potential energy landscape of a physical system evolving under Langevin dynamics. Standard statistical physics dictates 
that the probability density of such a system evolves deterministically according 
to the Fokker-Planck equation. However, solving the full Fokker-Planck equation is generally intractable for high-dimensional systems \cite{liu2022neural, sun2015numerical}. We show that for the 
quadratic potentials inherent to linear algebra, the system's Gaussian nature allows us to track the evolution of the autocorrelation matrix directly, reducing 
the problem to a tractable deterministic iteration. Table~\ref{tab:fokker_planck_comparison} summarizes the key computational distinctions between this standard PDE formulation 
and our direct covariance evolution derivation.} 

\begin{table}[htbp]
\centering
\caption{\bl{Key Distinctions: Fokker-Planck PDE versus Direct Covariance Evolution. 
The proposed direct method avoids the curse of dimensionality inherent in standard PDE solvers by exploiting the Gaussian structure of the problem to track only the statistical moments required for matrix inversion.}}
\label{tab:fokker_planck_comparison}
\resizebox{\textwidth}{!}{%
\begin{tabular}{lcc}
\toprule
\textbf{Attribute} & 
\begin{tabular}{@{}c@{}} \textbf{Standard Fokker-Planck} \\ \textbf{Approach} \end{tabular} & 
\begin{tabular}{@{}c@{}} \textbf{Our Direct Method} \\ \textbf{(This Work)} \end{tabular} \\ 
\toprule
\textbf{Computational Tractability} & 
\begin{tabular}{@{}c@{}} Intractable for high dimensions \\ (requires solving $n$-dimensional PDE) \end{tabular} & 
\begin{tabular}{@{}c@{}} Tractable matrix iteration \\ ($O(n^3)$ per step) \end{tabular} \\
\cmidrule(lr){1-3}
\textbf{Curse of Dimensionality} & 
\begin{tabular}{@{}c@{}} Exponential scaling $O(M^n)$ \\ (mesh discretization) \end{tabular} & 
\begin{tabular}{@{}c@{}} Polynomial scaling $O(n^3)$ \\ (matrix operations) \end{tabular} \\
\cmidrule(lr){1-3}
\textbf{What is Computed} & 
\begin{tabular}{@{}c@{}} Full probability density $P(x,t)$ \\ (complete distributional information) \end{tabular} & 
\begin{tabular}{@{}c@{}} Only covariance $\Sigma_t = \langle x \otimes x \rangle$ \\ (sufficient for matrix inversion) \end{tabular} \\
\cmidrule(lr){1-3}
\textbf{Gaussian Structure} & 
\begin{tabular}{@{}c@{}} Not exploited \\ (general PDE framework) \end{tabular} & 
\begin{tabular}{@{}c@{}} Directly exploited \\ (closed-form moment evolution) \end{tabular} \\
\cmidrule(lr){1-3}
\textbf{Practical Implementation} & 
\begin{tabular}{@{}c@{}} Specialized PDE solvers \\ (research-grade software) \end{tabular} & 
\begin{tabular}{@{}c@{}} Standard linear algebra libraries \end{tabular} \\
\bottomrule
\end{tabular}}
\end{table}

We prove that to first-order in the discretized dynamics, the autocorrelation matrix of the Ornstein-Uhlenbeck process at the heart of thermodynamic matrix 
inversion~\cite{aifer_thermodynamic_2024} is mathematically equivalent to preconditioned gradient descent on a carefully constructed loss function. 
This approximate equivalence demonstrates that complex stochastic thermal dynamics reduce to a relatively simple deterministic optimization process. \bl{While} physical 
implementations require thermal fluctuations to drive equilibration, the evolution of the autocorrelation matrix that encodes the solution follows a purely 
deterministic trajectory. This theoretical connection provides new insights into both the nature of thermodynamic computing and the relationship between 
physical processes and optimization algorithms, establishing a bridge between these seemingly disparate computational approaches and revealing shared 
mathematical foundations that were previously unrecognized.

By revealing the optimization structure underlying thermodynamic linear algebra~\cite{aifer_thermodynamic_2024}, we establish mathematical 
equivalences between approximate physical processes and optimization algorithms, gaining deeper insight into both domains. This theoretical perspective 
contributes to the broader understanding of how physical analogies can inspire computational methods, helping clarify when physical implementations 
offer unique advantages versus when they primarily serve as sources of algorithmic inspiration that can be realized through conventional digital means. 
Such theoretical analysis provides a framework for understanding the fundamental computational content of unconventional computing \bl{approaches~\cite{tye2023materials}}, 
bridging the gap between the engineering challenges of building novel hardware and the mathematical principles that govern their operation.

This article makes the following contributions to the state of the art:
\begin{enumerate}
    \item \textbf{New insight into thermodynamics-based SPD matrix inversion as a transformation of random variables.} The {\nameref{sec:ou_process} section} shows the link between 
    recent work on thermodynamics-based linear algebra in analog electrical systems and digital arithmetic on discretizations of probability density 
    functions. {Unlike standard stochastic simulations that typically rely on sampling Langevin dynamics, we demonstrate a method based on 
    deterministic arithmetic on discretized and compressed representations of arbitrary continuous probability density functions.} 
    This demonstrates that thermodynamic processes for computational problems can be faithfully implemented 
    on discrete-valued digital computing systems {without the sampling variability and convergence challenges of Monte-Carlo-type computations}.
    
    \item \textbf{First-order equivalence to preconditioned gradient descent optimization.} {While the general correspondence between Langevin 
    dynamics and deterministic probability evolution (Fokker-Planck) is established physics, {the \nameref{sec:gradient_descent} section} derives the specific 
    closed-form evolution for the covariance matrix in the context of matrix inversion. We prove when we perform a first-order approximation of the physical evolution with respect to the discrete time step, it is mathematically identical 
    to preconditioned gradient descent,} revealing that complex stochastic thermal dynamics for approximate SPD matrix inversion reduce to deterministic 
    iterative optimization and that the digital domain use of thermal fluctuations (use of randomness) are algorithmically {redundant} for this class of {convex} 
    problems when they are translated from their electrical analogs to the digital domain. This equivalence provides an alternative implementation 
    pathway that avoids the engineering challenges associated with analog hardware, such as the unavoidable tradeoff between ADC resolution, power dissipation, 
    and sample rate~\cite{meech_2023_data}, while maintaining the same mathematical foundations.

    {\item \textbf{Design implications for thermodynamic hardware.} We observe that the equivalence between first-order thermodynamic evolution and preconditioned gradient descent 
    implies specific criteria for efficient hardware realizations. This perspective suggests that the thermodynamic time constant corresponds to the optimal learning 
    rate $\Delta t^* = 1/(\lambda_{\max} + \lambda_{\min})$ and that physical parameters, such as coupling strengths and thermal reservoir properties, determine the 
    effective preconditioner for the target matrix $A$.}

    {\item \textbf{Historical connection to the von Neumann-Ulam method.} We outline the progression from the probabilistic sampling of the 1950s~\cite{forsythe_matrix_1950}, 
    through the physical analogies of thermodynamic computing, to the deterministic iteration presented here. We show that the spectral radius condition serves as the consistent 
    mathematical link between these approaches, demonstrating that the core mathematical content has remained constant even as the computational implementations have evolved.}
    
    \item \textbf{Schur complement acceleration.} {The \nameref{sec:numerical_experiments} section} demonstrates how the resulting optimization framework can be 
    significantly accelerated through Schur complement decompositions that exploit the natural block structure present in these linear algebraic 
    problems, providing practical speedups for the digital implementation. {We validate the presented algorithm against a digital thermodynamic 
    computing simulator, Thermox~\cite{duffield2024thermox}, and we obtain speedups in excess of 100,000-fold for certain problems.}
\end{enumerate}

{Our analysis clarifies the distinct role of stochastic noise in these computing paradigms by revealing a direct 
mathematical connection between analog thermodynamic processes and classical optimization algorithms. We find that for problems 
with a convex potential and a single global minimum, the essential computational dynamics are captured by a deterministic ordinary 
differential equation. This implies that the digital realization does not require the stochasticity that physical thermodynamic 
computing treats as a resource, as the same mathematical result is achieved through straightforward iterative optimization. 
This observation allows digital solvers to avoid the computational overhead of noise generation while retaining the 
algorithmic robustness of the framework. Whether similar optimization equivalences exist for other analog 
computing approaches remains an open question that merits further theoretical investigation.}

\bl{This analysis also contributes to a broader theoretical understanding of unconventional computing paradigms. By revealing the optimization structure underlying thermodynamic computing, 
we demonstrate how theoretical analysis can illuminate the mathematical foundations of physically-inspired computational methods. This suggests a general methodology for understanding 
unconventional computing approaches: identify the underlying optimization problem, analyze the physical dynamics as an analog solver, 
and extract the essential mathematical structure for digital implementation.}

\bl{This framework helps clarify when physical implementations offer unique advantages versus when they primarily serve as sources of algorithmic
inspiration that can be realized through conventional digital means. The theoretical perspective provides a bridge between the engineering 
challenges of building unconventional computing hardware and the mathematical principles that govern their operation, enabling more informed 
design decisions and performance predictions for thermodynamic computing systems.}

%
%

\section*{\bl{Results}}

\subsection*{First-order Equivalence to Preconditioned Gradient Descent}
\label{sec:gradient_descent}

We establish that the discretized Ornstein-Uhlenbeck process for SPD matrix inversion is equivalent to preconditioned gradient descent up to first-order in the discrete time step. We first derive the evolution equation for the \bl{covariance} matrix $\Sigma_i = \langle x_i \otimes x_i \rangle$ under the discretization scheme of Definition~\ref{def:discrete_ou}.

\begin{lemma}[Covariance Evolution Under Discretized \bl{O-U} Process]
\label{lem:autocorr_evolution}
{Let $A \in \mathbb{R}^{n \times n}$ be a symmetric positive definite matrix. Let $(x_i)_{i \geq 0}$ denote the Euler-Maruyama approximation of the Ornstein-Uhlenbeck process defined in Definition~\ref{def:discrete_ou}.
Then}, the \bl{covariance} matrix $\Sigma_i = \langle x_i \otimes x_i \rangle$ evolves according to:
\begin{equation}
\Sigma_{i+1} = (I - \Delta t A) \Sigma_{i} (I - \Delta t A) + 2\Delta t.
\end{equation}
\end{lemma}

\begin{proof}
We compute the \bl{covariance} at time step $i+1$:
\begin{align*}
\Sigma_{i+1} &= \langle x_{i+1} \otimes x_{i+1} \rangle \\
&= \left\langle \left[(I - \Delta t A) x_i + \sqrt{2\Delta t} Z_i\right] \otimes \left[(I - \Delta t A) x_i + \sqrt{2\Delta t} Z_i\right] \right\rangle.
\end{align*}
Expanding the outer product and using linearity of expectation:
\begin{align*}
\Sigma_{i+1} &= (I - \Delta t A)\langle x_i \otimes x_i \rangle(I - \Delta t A)^{\top} + \sqrt{2\Delta t}\langle Z_i \otimes x_i \rangle(I - \Delta t A)^{\top} \\
&\quad + (I - \Delta t A)\sqrt{2\Delta t}\langle x_i \otimes Z_i \rangle + 2\Delta t\langle Z_i \otimes Z_i \rangle.
\end{align*}
Since $x_i$ and $Z_i$ are independent, the cross terms vanish: $\langle x_i \otimes Z_i \rangle = \langle Z_i \otimes x_i \rangle = 0$. Moreover, $\langle Z_i \otimes Z_i \rangle = I$ since $Z_i$ has independent standard normal components, and $(I - \Delta t A)^{\top} = I - \Delta t A$ since $A$ is symmetric. Therefore:
\begin{equation}
\Sigma_{i+1} = (I - \Delta t A) \Sigma_{i} (I - \Delta t A) + 2\Delta t I.
\end{equation}
\end{proof}
This evolution equation reveals the underlying structure of the \bl{covariance} matrix of the discretized process. In order to prove our main result, we derive one more necessary algebraic property of the discretized process.

\begin{lemma}[Commutativity of $A$ and $\Sigma_{i}$]
\label{lem:commute}
Under the conditions of Lemma~\ref{lem:autocorr_evolution}, if the initial \bl{covariance} matrix $\Sigma_{0}$ commutes with $A$, then $\Sigma_{i}$ commutes with $A$ for all $i \geq 0$. In particular, if $\Sigma_{0} = \alpha I$ for some scalar $\alpha \geq 0$, then $A\Sigma_{i} = \Sigma_{i} A$ for all $i$.
\end{lemma}

\begin{proof}
We proceed by induction. The base case holds by assumption: $A\Sigma_{0} = \Sigma_{0} A$.

For the inductive step, assume $A\Sigma_{i} = \Sigma_{i} A$. From Lemma~\ref{lem:autocorr_evolution}:
\begin{equation}
\Sigma_{i+1} = (I - \Delta t A) \Sigma_{i} (I - \Delta t A) + 2\Delta t I.
\end{equation}

Since $\Sigma_{i}$ is a \bl{covariance} matrix of real random variables, it is symmetric (see \cite[p. 190]{papoulis1991}). We verify commutativity:
\begin{align*}
A\Sigma_{i+1} &= A[(I - \Delta t A) \Sigma_{i} (I - \Delta t A) + 2\Delta t I] \\
&= A(I - \Delta t A) \Sigma_{i} (I - \Delta t A) + 2\Delta t A \\
&= (I - \Delta t A) A \Sigma_{i} (I - \Delta t A) + 2\Delta t A,
\end{align*}
where we used that $A$ and $(I - \Delta t A)$ commute.

By the inductive hypothesis:
\begin{align}
A\Sigma_{i+1} &= (I - \Delta t A) \Sigma_{i} A (I - \Delta t A) + 2\Delta t A \\
&= (I - \Delta t A) \Sigma_{i} (I - \Delta t A) A + 2\Delta t A \\
&= [(I - \Delta t A) \Sigma_{i} (I - \Delta t A) + 2\Delta t I] A \\
&= \Sigma_{i+1} A.
\end{align}
Therefore, by induction, $A\Sigma_{i} = \Sigma_{i} A$ for all $i \geq 0$.
\end{proof}
We now present our central theoretical contribution: the digital implementation of the linearized thermodynamic matrix inversion reduces to preconditioned gradient descent on a carefully constructed loss function. This equivalence {establishes a direct mapping between the stochastic process and standard optimization machinery}.

\begin{theorem}[Approximate Ornstein-Uhlenbeck Process as Preconditioned Gradient Descent for SPD Matrix Inversion]
\label{theo:ou_gradient_descent}
Let $A \in \mathbb{R}^{n \times n}$ be symmetric positive definite and $x_0 = \mathbf{0}$. The discretized Ornstein-Uhlenbeck process produces iterative approximations to $A^{-1}$ via preconditioned gradient descent. Specifically, the \bl{covariance} matrix $\Sigma_i$, which approximates $A^{-1}$, evolves according to
\begin{equation}
\label{eq:grad_desc}
\Sigma_{i+1} = \Sigma_i - 2\Delta t (\Sigma_i A - I)+O\left(\Delta t^2 \right).
\end{equation}
Upon neglecting quadratic terms of the time step, this update rule is precisely preconditioned gradient descent minimizing $L(\Sigma) = \|\Sigma A - I\|_F^2$ with preconditioner $P = A$ and learning rate $\Delta t$.
\end{theorem}

\begin{proof}
From Lemmas~\ref{lem:autocorr_evolution} and \ref{lem:commute}, we have
\begin{equation*}
\Sigma_{i+1} = (I - \Delta t A)^2 \Sigma_i + 2\Delta t I.
\end{equation*}
Expanding to first-order in $\Delta t$:
\begin{align*}
\Sigma_{i+1} &= (I - 2\Delta t A + O\left(\Delta t^2 \right)) \Sigma_i + 2\Delta t I \\
&= \Sigma_i - 2\Delta t A \Sigma_i + 2\Delta t I + O\left(\Delta t^2 \right) \\
&= \Sigma_i - 2\Delta t (\Sigma_i A - I) + O\left(\Delta t^2 \right),
\end{align*}
where the last equality follows from the commutativity $A\Sigma_i = \Sigma_i A$, since $\Sigma_0 = \alpha I$ for $\alpha = 0$.
For the loss function $L(\Sigma) = \|\Sigma A - I\|_F^2$, the gradient is
\begin{equation*}
\nabla L = 2(\Sigma A - I)A.
\end{equation*}
The preconditioned gradient descent update with preconditioner $P$ is
\begin{equation*}
\Sigma_{i+1} = \Sigma_i - \eta P^{-1} \nabla L.
\end{equation*}
Setting $P = A$ and $\eta = \Delta t$ yields
\begin{align*}
\Sigma_{i+1} &= \Sigma_i - \Delta t A^{-1} \cdot 2(\Sigma_i A - I)A \\
&= \Sigma_i - 2\Delta t (\Sigma_i A - I),
\end{align*}
which coincides with our derived update rule up to first-order in $\Delta t$.
\end{proof}

\begin{figure}[!t]
\centering
\includegraphics[width=1.0\textwidth]{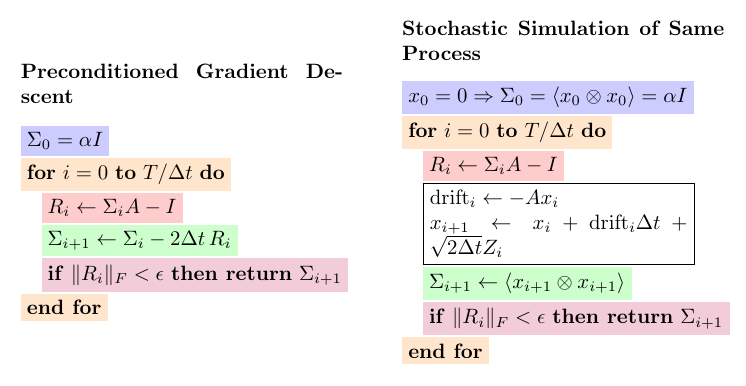}

\label{fig:true_equivalence}
\end{figure}

\begin{remark}
This theorem establishes an interesting connection between stochastic processes and optimization algorithms, and also opens pathways for potential algorithmic acceleration.
\end{remark}

\begin{remark}[Comparison with Newton-Schulz Algorithm]
Equation~\eqref{eq:grad_desc} bears a notable resemblance to the classical Newton-Schulz iteration \cite{schulz_iterative_1933, hotelling_analysis_1943}. While Newton-Schulz employs the update
\begin{equation*}\tilde{A}^{-1}_{i+1} = \tilde{A}^{-1}_i(2I - A\tilde{A}^{-1}_i) = 2\tilde{A}^{-1}_i - \tilde{A}^{-1}_i A \tilde{A}^{-1}_i,\end{equation*}
Theorem~\ref{theo:ou_gradient_descent} uses the simpler update
\begin{equation*}\tilde{A}^{-1}_{i+1} = \tilde{A}^{-1}_i - 2\Delta t \left(\tilde{A}^{-1}_i A - I\right).\end{equation*}
\end{remark}

The key advantages of the gradient descent formulation are: (1) \textit{Computational simplicity}: Each iteration requires only one matrix multiplication $\tilde{A}^{-1}_i A$ and one matrix subtraction, compared to Newton-Schulz's requirement for two matrix multiplications per iteration. (2) \textit{Direct optimization}: The algorithm explicitly minimizes the residual $\|R_i\|_F = \|\tilde{A}^{-1}_i A - I\|_F$, making the objective transparent. Nevertheless, the Newton-Schulz algorithm bears the advantage that it is not dependent on a learning rate, as is the case with the preconditioned gradient descent.

\subsection*{Historical Connection to the von Neumann-Ulam Method}

The convergence analysis of the iterative update~\eqref{eq:grad_desc} 
reveals a connection to the earliest Monte Carlo approach to matrix inversion, 
developed by von Neumann and Ulam 
and published by Forsythe and Leibler in 1950~\cite{forsythe_matrix_1950}. 
This historical parallel demonstrates the mathematical continuity 
underlying seemingly disparate computational approaches.

Forsythe and Leibler's method exploits the Neumann series expansion 
$B^{-1} = (I - A)^{-1} = \sum_{k=0}^{\infty} A^k$ where $A = I - B$, 
requiring the convergence condition 
$\max_r |1 - \lambda_r(B)| = \max_r |\lambda_r(A)| < 1$. 
Their ``solitaire game'' samples random walks through this infinite series, 
where each walk contributes to estimating individual matrix elements 
through Monte Carlo averaging.

Our gradient descent algorithm, while appearing different, 
operates on the same mathematical foundation. 
The error evolution $E_{i+1} = E_i(I - 2\Delta t A)$ 
implicitly computes the series $(I - 2\Delta t A)^i E_0$, 
which converges to zero when $\rho(I - 2\Delta t A) < 1$. 
This spectral radius condition $\rho(I - 2\Delta t A) < 1$ 
is mathematically equivalent to Forsythe and Leibler's requirement $\rho(A) < 1$ 
when $\Delta t$ is appropriately chosen.

The historical trajectory from 1950 to present 
shows an evolution of the same core mathematics:
\begin{enumerate}
\item \textbf{1950}: von Neumann and Ulam conceptualize matrix inversion 
as Monte Carlo sampling of the Neumann series
\item \textbf{2024}: Aifer \et\ connect SPD matrix inversion 
to thermodynamic equilibration via Ornstein-Uhlenbeck processes
\item \textbf{Present work}: We demonstrate that both stochastic approaches 
reduce to elementary gradient descent on the deterministic objective $\|XA - I\|_F^2$
\end{enumerate}

This progression represents a complete circle in computational thinking: 
from probabilistic sampling of the 1950s, 
through physical analogies of thermodynamic computing, 
to simple deterministic iteration that captures the essential mathematical content. 
The spectral radius serves as the unifying mathematical thread 
connecting these disparate computational paradigms, 
where the mathematics has remained constant 
while our understanding of how to exploit it has evolved.

\subsection*{Convergence Analysis}

The choice of learning rate $\Delta t$ in Theorem~\ref{theo:ou_gradient_descent} 
is critical for both convergence and computational efficiency.
{We now establish theoretical bounds for selecting appropriate step sizes, 
following classical convergence analysis
for gradient-based methods~\cite{boyd_convex_2004, nocedal_numerical_2006}.}

To analyze convergence, we examine the error evolution. 
Let $E_i = \tilde{A}^{-1}_i - A^{-1}$ denote the error at iteration $i$. 
Substituting $\tilde{A}^{-1}_i = E_i + A^{-1}$ into the update rule:

\begin{align*}
\tilde{A}^{-1}_{i+1} &= \tilde{A}^{-1}_i - 2\Delta t (\tilde{A}^{-1}_i A - I) \\
A^{-1} + E_{i+1} &= A^{-1} + E_i - 2\Delta t ((A^{-1} + E_i) A - I) \\
E_{i+1} &= E_i - 2\Delta t E_i A = E_i(I - 2\Delta t A).
\end{align*}

The error evolution is governed by the linear transformation $I - 2\Delta t A$, 
whose convergence properties are determined by 
spectral analysis~\cite{horn_matrix_2012, golub_matrix_2013}.

\begin{theorem}[Convergence Condition]
\label{thm:convergence_condition}
The iteration~\eqref{eq:grad_desc} converges to $A^{-1}$ 
if and only if the spectral radius $\rho(I - 2\Delta t A) < 1$~\cite{varga_matrix_2000}, 
which occurs when:
\begin{equation*}
0 < \Delta t < \frac{1}{\lambda_{\max}(A)},
\end{equation*}
where $\lambda_{\max}(A)$ is the largest eigenvalue of $A$.
\end{theorem}

\begin{proof}
Let $\lambda_1, \lambda_2, \ldots, \lambda_n$ be the eigenvalues of $A$. 
The eigenvalues of $I - 2\Delta t A$ are $1 - 2\Delta t \lambda_j$ 
for $j = 1, \ldots, n$.

For convergence, we require $|1 - 2\Delta t \lambda_j| < 1$ 
for all $j$~\cite{saad_iterative_2003}:
\begin{align*}
-1 &< 1 - 2\Delta t \lambda_j < 1 \\
-2 &< -2\Delta t \lambda_j < 0 \\
0 &< \Delta t \lambda_j < 1.
\end{align*}
Since $A$ is positive definite, all $\lambda_j > 0$. 
The most restrictive condition is:
$$\Delta t < \frac{1}{\max_j \lambda_j} = \frac{1}{\lambda_{\max}(A)}.$$
\end{proof}

\begin{theorem}[Optimal Step Size]
\label{thm:optimal_step_size}
The learning rate that minimizes the spectral radius~\cite{polyak_introduction_1987} 
of $I - 2\Delta t A$ is:
\begin{equation*}
\Delta t^* = \frac{1}{\lambda_{\max}(A) + \lambda_{\min}(A)},
\end{equation*}
with a spectral radius of:
\begin{equation*}
\rho(I - 2\Delta t^* A) = \frac{\kappa(A) - 1}{\kappa(A) + 1},
\end{equation*}
where $\kappa(A) = \lambda_{\max}(A)/\lambda_{\min}(A)$ 
is the condition number~\cite{higham_accuracy_2002} of $A$ 
and $\rho(\cdot)$ is the spectral radius.
\end{theorem}

\begin{proof}
The eigenvalues of $I - 2\Delta t A$ lie in the interval:
$$[1 - 2\Delta t \lambda_{\max}, 1 - 2\Delta t \lambda_{\min}].$$

To minimize the spectral radius, 
we seek the step size where the extreme eigenvalues are symmetric around zero. 
This occurs when:
$$1 - 2\Delta t \lambda_{\max}(A) = -(1 - 2\Delta t \lambda_{\min}(A)).$$

This condition ensures that the two extreme eigenvalues have equal absolute values, 
minimizing $\rho = \max\{|1 - 2\Delta t \lambda_{\max}|, |1 - 2\Delta t \lambda_{\min}|\}$.

Solving for $\Delta t$:
$$1 - 2\Delta t \lambda_{\max}(A) = -1 + 2\Delta t \lambda_{\min}(A)$$,
$$\Delta t^* = \frac{1}{\lambda_{\max}(A) + \lambda_{\min}(A)}$$

The resulting spectral radius is:
$$\rho = |1 - 2\Delta t^* \lambda_{\max}(A)| 
= \left|1 - \frac{2\lambda_{\max}(A)}{\lambda_{\max}(A) + \lambda_{\min}(A)}\right| 
= \frac{\lambda_{\max}(A) - \lambda_{\min}(A)}{\lambda_{\max}(A) + \lambda_{\min}(A)} 
= \frac{\kappa(A) - 1}{\kappa(A) + 1}.$$
\end{proof}

\begin{remark}
\label{rem:rate}
The optimal convergence rate $(\kappa(A)-1)/(\kappa(A)+1)$ shows that 
well-conditioned matrices ($\kappa(A) \approx 1$) converge rapidly, 
while ill-conditioned matrices ($\kappa(A) \gg 1$) converge slowly, 
approaching the bound $\rho \to 1$. 
This behavior is characteristic of first-order methods 
applied to quadratic functions~\cite{nesterov_introductory_2004}.
\end{remark}

\subsection*{Numerical Experiments}\label{sec:numerical_experiments}

The experimental evaluation serves two primary purposes: 
validating the theoretical equivalence we establish 
and evaluating the practical performance of the resulting deterministic algorithm 
against established benchmarks.
We focus on showing that the simplified deterministic algorithm 
faithfully captures the thermodynamic convergence behavior
while maintaining competitive performance with established matrix inversion methods. 
We emphasize, however, that our primary
objective is to demonstrate the validity of the thermodynamic-to-deterministic mapping. 
We do not claim to displace highly-optimized numerical linear algebra libraries
(e.g., LAPACK) for general-purpose computing, 
but rather to show that the physics-inspired algorithm 
belongs to the same efficiency class as standard iterative methods.

To isolate the algorithmic performance from learning rate selection effects 
and provide a best-case analysis,
we employ the theoretically optimal learning rate 
$\Delta t^* = 1/(\lambda_{\max} + \lambda_{\min})$ throughout our experiments. 
While computing eigenvalues defeats the purpose of iterative methods in practice, 
this choice allows us to demonstrate the fundamental performance characteristics 
of the gradient descent formulation when operating under optimal conditions. 
This approach provides an upper bound on performance 
that practical adaptive learning rate schemes would aim to approximate. {We use 
the theoretically optimal learning rate as a controlled experimental condition 
to enable fair comparison between algorithms 
and to demonstrate achievable performance bounds, 
not as a practical implementation recommendation.}

We structure our experiments to address three key questions: 
(1) Does the deterministic formulation accurately reproduce 
the convergence behavior of the stochastic thermodynamic process? 
(2) How does the derived first-order algorithm compare 
to established second-order benchmarks like Newton-Schulz? 
(3) What role does learning rate selection play in practical performance? 
To provide context for these comparisons, 
Table~\ref{tab:applications} summarizes common applications requiring matrix inversion 
and their typical matrix properties. 
These applications motivate our focus on symmetric positive definite matrices 
and highlight the role of the condition numbers in these applications.

\begin{table*}
  \caption{\bl{Applications of Matrix Inversion and Expected Matrix Properties.
       We survey ten major application domains where matrix inversion is required.
       Here $p$ denotes dimensionality, $n$ denotes sample size, and $h$ denotes discretization mesh size.
       SPD indicates whether matrices are symmetric positive definite.
       Normal Eqs indicates whether the application requires forming normal equations
       to obtain an SPD matrix.}}
\label{tab:applications}
\resizebox{\textwidth}{!}{%
\begin{tabular}{llcccp{4.5cm}}
\toprule
\multicolumn{6}{c}{\textbf{Part 1: Matrix Properties and Context}} \\
\midrule
\textbf{Application} & \textbf{Matrix Type} & \textbf{SPD?} & \textbf{Normal Eqs?} & \textbf{Reference} & \textbf{Notes} \\
\midrule
   Bayesian inference & Precision matrices & Yes & No & \cite{leonard_bayesian_1992} & \\
   Linear regression & Gram matrix $(X^T X)$ & Yes & Yes & \cite{golub_matrix_2013} & When $X$ has full column rank \\
   Portfolio optimization & Asset \bl{covariance} matrices & Yes & No & \cite{markowitz_portfolio_1952} & \\
   Data assimilation & Observation error \bl{covariance} & Yes & Sometimes & \cite{daley_atmospheric_1991, tabeart_improving_2019} & \\
   Econometrics (GMM) & Weighting matrices & Yes & No & \cite{hansen_large_1982} & \\
   Mahalanobis distance & Covariance matrix & Yes & No & \cite{mahalanobis_generalised_1936} & \\
   Medical imaging & Measurement operators & No & Yes & \cite{cahill_preliminary_1970} & \\
   Computational physics & Discretized operators & Depends & Sometimes & \cite{franklin_matrix_2013} & SPD for elliptic PDEs; not for hyperbolic/parabolic \\
   Preconditioners & Approximation matrices & Usually & No & \cite{saad_iterative_2003} & \\
   Covariance estimation & Sample \bl{covariance} & Yes & No & \cite{ledoit_well-conditioned_2004} & When $n > p$ \\
\midrule
\multicolumn{6}{c}{\textbf{Part 2: Condition Number Scaling}} \\
\midrule
\textbf{Application} & \multicolumn{5}{c}{\textbf{Condition Number Scaling}} \\
\midrule
   Bayesian inference & \multicolumn{5}{l}{Depends on parameter correlations} \\
   Linear regression & \multicolumn{5}{l}{$\kappa \propto$ feature correlations (multicollinearity)} \\
   Portfolio optimization & \multicolumn{5}{l}{$\kappa$ increases as $p \to n$} \\
   Data assimilation & \multicolumn{5}{l}{$\kappa > 10^5$ due to interchannel correlations} \\
   Econometrics (GMM) & \multicolumn{5}{l}{$\kappa$ increases with moment conditions} \\
   Mahalanobis distance & \multicolumn{5}{l}{$\kappa$ increases as $p \to n$} \\
   Medical imaging & \multicolumn{5}{l}{$\kappa$ increases with undersampling} \\
   Computational physics & \multicolumn{5}{l}{$\kappa = O(h^{-2})$ for elliptic PDEs} \\
   Preconditioners & \multicolumn{5}{l}{Designed to be well-conditioned} \\
   Covariance estimation & \multicolumn{5}{l}{$\kappa$ increases as $p \to n$} \\
\bottomrule
\end{tabular}}
\end{table*}

To ensure controlled and reproducible testing conditions throughout all experiments, 
we generate test matrices with predetermined condition numbers 
using spectral decomposition. 
For each test configuration, 
we construct symmetric positive definite matrices as $A = QDQ^T$, 
where we obtain $Q$ as an orthogonal matrix 
via QR decomposition of random Gaussian matrices, 
and we design $D$ as a diagonal matrix with eigenvalues 
that yield the target condition number $\kappa$. 
We measure the approximation error of each method 
as the Frobenius norm of the residual matrix $\|XA - I\|_F$, 
where $X$ represents the computed approximate inverse 
and $I$ denotes the identity matrix. 
This metric quantifies how closely the computed solution satisfies $XA = I$.

All experiments were implemented in Python on a MacBook Pro 
with 8-GB memory running on an Apple M3 processor, 
using NumPy \cite{harris2020array} for numerical computations 
and SciPy \cite{2020SciPy-NMeth} for eigenvalue calculations.

{We first validate our deterministic formulation 
against stochastic thermodynamic simulation.}
Thermox \cite{duffield2024thermox} is a digital software package 
developed by Duffield and Donatella at Normal Computing,
the same research group behind the original thermodynamic linear algebra work 
of Aifer \et~\cite{aifer_thermodynamic_2024}.
Unlike traditional stochastic differential equation solvers 
that rely on time-discretization approximations (such as Euler-Maruyama),
Thermox utilizes analytical solutions for the mean and covariance 
of the Ornstein-Uhlenbeck process to generate samples without discretization error.
The authors position Thermox as a rigorous digital proxy 
for ideal thermodynamic hardware.

The comparison against Thermox \cite{duffield2024thermox} 
serves as a numerical validation of our theoretical derivation.
Thermox implements the stochastic simulation (Langevin integration) approach 
to thermodynamic computing described by Aifer \et~\cite{aifer_thermodynamic_2024}, 
making it the ideal platform for verifying 
that our deterministic algorithm faithfully reproduces 
the convergence behavior of the thermodynamic method. 
This comparison directly validates our central theoretical result:
that the linearized stochastic thermal dynamics are mathematically equivalent 
to deterministic gradient descent 
and can be simulated without the sampling overhead.
In particular, the iterative update~\eqref{eq:grad_desc} 
is mathematically equivalent to the stochastic thermodynamic algorithm 
from Aifer \et \cite{aifer_thermodynamic_2024} to first-order in the Euler-Maruyama discretization,
with the stochastic machinery removed. 
See Figure 5 for an algorithmic illustration.

We configure both our Gradient Descent (GD) algorithm and Thermox 
to use identical conditions for fair comparison.
We set the gradient descent learning rate 
to the theoretically-optimal value $\Delta t^* = 1/(\lambda_{\max} + \lambda_{\min})$ 
and apply the same step size to the thermodynamic sampling process. 
As noted earlier, we pre-compute these optimal rates 
to provide a best-case performance comparison 
rather than as a practical implementation strategy. 
We fix the matrix dimension at $n = 20$ and condition number at $\kappa = 20$. 
For the thermodynamic approach, 
we vary the number of samples from $10^3$ to $10^7$ 
and average results over 50 independent runs to account for stochastic variability. 
This design allows us to assess the computational cost 
inherent to stochastic simulation compared to the direct moment evolution.

\begin{figure}[!t]
\centering
\includegraphics[width=1.0\textwidth]{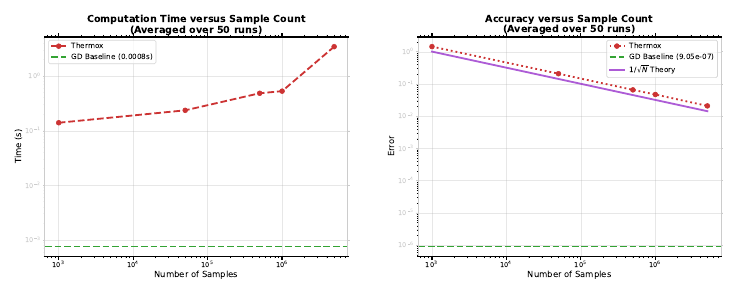}
\label{fig:samples}
\end{figure}

Figure 6 presents two plots 
that show how increasing sample counts 
affect the thermodynamic algorithm's computational performance.
The left plot shows computation time as a function of sample count, 
with the GD baseline time shown as a horizontal reference line.
The right plot shows approximation error versus number of samples, 
along with the theoretical Langevin dynamics simulation convergence rate 
of $1/\sqrt{N_s}$.
These plots illustrate that 5,000,000 samples, 
which require 3048ms to simulate, 
fail to achieve accuracy within four orders of magnitude of GD,
which executes in 0.8ms. 
This confirms the theoretical prediction of Theorem~\ref{theo:ou_gradient_descent}, 
showing that the stochastic sampling process 
is computationally redundant for the task of matrix inversion.

To illustrate the computational implications of this first-order equivalence, 
we fix the number of Thermox samples at 1,000,000 
and set the convergence tolerance for GD to $\epsilon = 10^{-2}$ 
to match the accuracy of the thermodynamic process as closely as possible. 
The rest of the experimental setup remains identical to the previous example. 
Table~\ref{tab:gd_vs_therm} presents timing and approximation error comparisons 
between Thermox and GD across matrix sizes $n \in \{10, 20, 50, 100\}$ 
and condition numbers $\kappa \in \{2, 50, 500\}$, 
yielding 12 distinct test configurations. 
In all 12 experiments, 
the deterministic variant achieves orders of magnitude performance improvements, 
confirming our theoretical predictions. 
We emphasize that this comparison serves primarily 
to validate Theorem~\ref{theo:ou_gradient_descent}, 
demonstrating that the deterministic moment evolution 
faithfully captures the thermodynamics without the sampling overhead, 
rather than to establish a performance benchmark. 
Simulating stochastic differential equations 
is inherently more expensive than deterministic iteration; 
the result confirms that the thermodynamic \textit{algorithm} 
is mathematically equivalent to the deterministic one, 
rendering the stochastic simulation redundant for this task.

\begin{table*}
  \caption{\bl{Gradient Descent versus Thermox Comparison with Optimal Learning Rate.
     GD consistently achieves lower error than Monte Carlo sampling (Thermox) for approximate matrix inversion,
     with speedups ranging from $54\times$ to $166{,}000\times$ across dimensions 10--100 and condition numbers 2--500.}}
  \label{tab:gd_vs_therm}
  \resizebox{\textwidth}{!}{%
  \begin{tabular}{ccccccccc}
  \toprule
   \textbf{Dim} &  \textbf{Cond} &  \begin{tabular}{@{}c@{}}\textbf{Iterations/} \\ \textbf{Samples (GD)}\end{tabular} &  \begin{tabular}{@{}c@{}}\textbf{Iterations/} \\ \textbf{Samples (Thermox)}\end{tabular} &  \begin{tabular}{@{}c@{}}\textbf{Time} \\ \textbf{GD (s)}\end{tabular} &  \begin{tabular}{@{}c@{}}\textbf{Time} \\ \textbf{Thermox (s)}\end{tabular} &  \begin{tabular}{@{}c@{}}\textbf{Error} \\ \textbf{GD}\end{tabular} &  \begin{tabular}{@{}c@{}}\textbf{Error} \\ \textbf{Thermox}\end{tabular} &  \begin{tabular}{@{}c@{}}\textbf{Speedup} \\ \textbf{($\times$)}\end{tabular} \\
  \midrule

   10 & 2 & 5 & 1000000 & 0.0000 & 7.6459 & 5.16e-03 & 1.69e-02 & 166161.77 \\
   10 & 50 & 151 & 1000000 & 0.0005 & 0.4757 & 9.76e-03 & 4.44e-02 & 913.12 \\
   10 & 500 & 2125 & 1000000 & 0.0071 & 0.3837 & 9.97e-03 & 9.44e-02 & 53.83 \\
  \midrule

   20 & 2 & 5 & 1000000 & 0.0000 & 6.5576 & 5.97e-03 & 3.17e-02 & 137522.63 \\
   20 & 50 & 151 & 1000000 & 0.0008 & 0.7071 & 9.76e-03 & 6.04e-02 & 941.49 \\
   20 & 500 & 2125 & 1000000 & 0.0105 & 0.5518 & 9.97e-03 & 1.49e-01 & 52.53 \\
  \midrule

   50 & 2 & 5 & 1000000 & 0.0004 & 7.9436 & 8.23e-03 & 7.50e-02 & 21876.46 \\
   50 & 50 & 151 & 1000000 & 0.0015 & 1.8516 & 9.76e-03 & 1.34e-01 & 1255.43 \\
   50 & 500 & 2125 & 1000000 & 0.0231 & 1.7854 & 9.97e-03 & 3.00e-01 & 77.35 \\
  \midrule

   100 & 2 & 6 & 1000000 & 0.0002 & 10.0047 & 3.44e-03 & 1.54e-01 & 54925.04 \\
   100 & 50 & 151 & 1000000 & 0.0037 & 4.8165 & 9.77e-03 & 2.55e-01 & 1288.87 \\
   100 & 500 & 2125 & 1000000 & 0.0722 & 4.3894 & 9.97e-03 & 4.51e-01 & 60.80 \\

  \bottomrule
  \end{tabular}}
\end{table*}

{Having validated the equivalence, 
we now evaluate practical performance against established deterministic methods.}
While Krylov subspace methods such as Preconditioned Conjugate Gradient (PCG) 
are often the gold standard for solving linear systems ($Ax=b$)
involving SPD matrices~\cite{saad_iterative_2003, golub_matrix_2013}, 
they are less commonly used for inverting a matrix $A \in \mathbb{R}^{n \times n}$,
as they would require $n$ independent solves. 
For explicit inversion, 
the most appropriate direct comparison 
is the Newton-Schulz (or Newton-Schulz) iteration~\cite{schulz_iterative_1933, hotelling_analysis_1943}, 
a standard Newton-based method.

From a numerical linear algebra perspective, 
our derived deterministic update~\eqref{eq:grad_desc} 
corresponds to a first-order preconditioned Richardson iteration~\cite{saad_iterative_2003}, 
whereas Newton-Schulz is a second-order Newton method. 
This comparison therefore tests 
whether the deterministic algorithm extracted from thermodynamic principles 
offers competitive performance against the classical ``digital'' standard 
for this specific task.

The experimental parameter space encompasses matrix dimensions 
$n \in \{8, 32, 128, 512, 1024\}$ and condition numbers $\kappa \in \{2, 50, 500\}$, 
yielding 15 distinct test configurations. 
Both algorithms employ a convergence tolerance $\epsilon = 10^{-3}$ 
with a maximum iteration limit of 10,000. 
For the gradient descent method, 
we utilize the theoretically optimal learning rate 
$\Delta t^* = \frac{1}{\lambda_{\max} + \lambda_{\min}}$, 
while the Newton-Schulz algorithm employs its standard initialization scheme.

Table~\ref{tab:detailed_results} presents the comparative results 
across all test configurations. 
The data reveal a clear performance dichotomy based on matrix conditioning. 
For well-conditioned matrices ($\kappa = 2$), 
GD demonstrates superior performance in terms of iteration count, 
and competitive performance in terms of computational time. 
However, as the condition number increases to $\kappa = 50$ and $\kappa = 500$, 
the Newton-Schulz method exhibits substantially better convergence characteristics, 
requiring significantly fewer iterations to achieve the same accuracy tolerance. 
This behavior aligns with the theoretical analysis in Remark~\ref{rem:rate}, 
which predicts that gradient descent performance degrades 
with increasing condition number 
while Newton-Schulz maintains more stable convergence properties 
for ill-conditioned systems.

\begin{table*}
  \caption{\bl{Detailed Comparison of Gradient Descent (GD) versus Schulz-Hotelling (SH) Iteration for Approximate Matrix Inversion.
       We compare convergence behavior across varying matrix dimensions (8--1024) and condition numbers (2--500) with optimal learning rates.
       SH demonstrates superior performance for ill-conditioned matrices, requiring fewer iterations and achieving lower error.
       For well-conditioned matrices, GD shows competitive performance at smaller dimensions but SH dominates at larger scales.
       Speedup represents the ratio of GD execution time to SH execution time.}}
  \label{tab:detailed_results}
  \resizebox{\textwidth}{!}{%
  \begin{tabular}{ccccccccc}
  \toprule
   \textbf{Dim} &  \textbf{Cond} &  \begin{tabular}{@{}c@{}}\textbf{Iterations} \\ \textbf{GD}\end{tabular} &  \begin{tabular}{@{}c@{}}\textbf{Iterations} \\ \textbf{SH}\end{tabular} &  \begin{tabular}{@{}c@{}}\textbf{Time} \\ \textbf{GD (s)}\end{tabular} &  \begin{tabular}{@{}c@{}}\textbf{Time} \\ \textbf{SH (s)}\end{tabular} &  \begin{tabular}{@{}c@{}}\textbf{Error} \\ \textbf{GD}\end{tabular} &  \begin{tabular}{@{}c@{}}\textbf{Error} \\ \textbf{SH}\end{tabular} &  \begin{tabular}{@{}c@{}}\textbf{Speedup} \\ \textbf{($\times$)}\end{tabular} \\
  \midrule

   8 & 2.0 & 7 & 7 & 0.0001 & 0.0004 & 5.53e-04 & 9.41e-04 & 4.00 \\
   8 & 50.0 & 208 & 16 & 0.0007 & 0.0001 & 9.98e-04 & 1.20e-04 & 0.14 \\
   \bl{8} & \bl{100.0} & \bl{456} & \bl{18} & \bl{0.0018} & \bl{0.0001} & \bl{9.90e-04} & \bl{1.12e-04} & \bl{0.04} \\
   8 & 500.0 & 2700 & 23 & 0.0084 & 0.0001 & 1.00e-03 & 8.09e-06 & 0.01 \\
  \midrule

   32 & 2.0 & 7 & 10 & 0.0002 & 0.0001 & 6.98e-04 & 1.14e-06 & 0.50 \\
   32 & 50.0 & 208 & 18 & 0.0010 & 0.0002 & 9.98e-04 & 7.56e-05 & 0.20 \\
   \bl{32} & \bl{100.0} & \bl{456} & \bl{20} & \bl{0.0026} & \bl{0.0003} & \bl{9.90e-04} & \bl{6.89e-05} & \bl{0.04} \\
   32 & 500.0 & 2700 & 25 & 0.0176 & 0.0003 & 1.00e-03 & 4.29e-06 & 0.02 \\
  \midrule

   128 & 2.0 & 8 & 12 & 0.0011 & 0.0010 & 3.79e-04 & 1.82e-06 & 0.91 \\
   128 & 50.0 & 208 & 20 & 0.0106 & 0.0014 & 9.98e-04 & 6.78e-05 & 0.13 \\
   \bl{128} & \bl{100.0} & \bl{456} & \bl{22} & \bl{0.0142} & \bl{0.0014} & \bl{9.90e-04} & \bl{6.16e-05} & \bl{0.05} \\
   128 & 500.0 & 2700 & 27 & 0.1018 & 0.0018 & 1.00e-03 & 3.69e-06 & 0.02 \\
  \midrule

   512 & 2.0 & 8 & 14 & 0.0199 & 0.0524 & 7.23e-04 & 3.43e-06 & 2.63 \\
   512 & 50.0 & 211 & 22 & 0.2519 & 0.0607 & 9.83e-04 & 6.68e-05 & 0.24 \\
   \bl{512} & \bl{100.0} & \bl{457} & \bl{24} & \bl{0.5794} & \bl{0.0753} & \bl{9.84e-04} & \bl{5.99e-05} & \bl{0.05} \\
   512 & 500.0 & 2700 & 29 & 3.1689 & 0.0756 & 1.00e-03 & 3.56e-06 & 0.02 \\
  \midrule

   1024 & 2.0 & 9 & 15 & 0.0936 & 0.3437 & 3.20e-04 & 4.79e-06 & 3.67 \\
   1024 & 50.0 & 215 & 23 & 2.0007 & 0.5270 & 9.98e-04 & 7.13e-05 & 0.26 \\
   \bl{1024} & \bl{100.0} & \bl{460} & \bl{25} & \bl{4.8963} & \bl{0.6140} & \bl{1.00e-03} & \bl{6.02e-05} & \bl{0.05} \\
   1024 & 500.0 & 2700 & 30 & 28.9243 & 0.6729 & 1.00e-03 & 3.54e-06 & 0.02 \\

  \bottomrule
  \end{tabular}}
\end{table*}

{The Schur complement provides a divide-and-conquer approach 
that can accelerate both methods for large matrices,
especially those with natural block structure.}
The recursive algorithm partitions matrices into blocks 
and applies the Schur complement relationship defined in Definition~\ref{schurDefinition}. 
Since the Schur complement of an SPD matrix remains SPD \cite[Section 7.7.5]{horn_matrix_2013},
and the block $D$ in Definition~\ref{schurDefinition} of an SPD matrix is also SPD, 
the recursive decomposition maintains favorable numerical properties 
throughout the computation.

The method's effectiveness comes from balancing 
the advantages of hierarchical decomposition with computational efficiency. 
By recursively splitting matrices until reaching a matrix size $\tau$,
chosen a priori, 
the algorithm exhibits divide-and-conquer benefits for large matrices 
while switching to the preconditioned gradient descent approach 
presented in the {\nameref{sec:gradient_descent} section}
for the inner approximate inversion calls. 
This hybrid strategy takes advantage of the guaranteed positive eigenvalues 
and typically well-conditioned nature of small SPD matrices. 
The complete accelerated matrix inversion algorithm 
is presented in Algorithm~\ref{alg:recursive_schur_inversion}.

\begin{algorithm}
\caption{Recursive Schur Complement Matrix Inversion}
\label{alg:recursive_schur_inversion}
\begin{algorithmic}
\REQUIRE Matrix $A \in \mathbb{R}^{n \times n}$, minimum size to split $\tau$, GD parameters $\Delta t$, $K$, $\epsilon$
\ENSURE Approximate inverse $A^{-1}$
\STATE \textsc{SchurInvert}($A$, $\tau$, $\Delta t$, $K$, $\epsilon$):
\STATE \hspace{\algorithmicindent} \textbf{if} $\text{size}(A) \leq \tau$ \textbf{then}
\STATE \hspace{\algorithmicindent} \hspace{\algorithmicindent} \textbf{return} \textsc{GDApproximateInverse}($A$, $\Delta t$, $K$, $\epsilon$)
\STATE \hspace{\algorithmicindent} \textbf{else}
\STATE \hspace{\algorithmicindent} \hspace{\algorithmicindent} Split $A$ into blocks: $A = \begin{bmatrix} A_{11} & A_{12} \\ A_{21} & A_{22} \end{bmatrix}$
\STATE \hspace{\algorithmicindent} \hspace{\algorithmicindent} $A_{22}^{-1} \leftarrow$ \textsc{SchurInvert}($A_{22}$, $\tau$, $\Delta t$, $K$, $\epsilon$)
\STATE \hspace{\algorithmicindent} \hspace{\algorithmicindent} $S \leftarrow A_{11} - A_{12} A_{22}^{-1} A_{21}$
\STATE \hspace{\algorithmicindent} \hspace{\algorithmicindent} $S^{-1} \leftarrow$ \textsc{SchurInvert}($S$, $\tau$, $\Delta t$, $K$, $\epsilon$)
\STATE \hspace{\algorithmicindent} \hspace{\algorithmicindent} $B_{11} \leftarrow S^{-1}$
\STATE \hspace{\algorithmicindent} \hspace{\algorithmicindent} $B_{12} \leftarrow -S^{-1} A_{12} A_{22}^{-1}$
\STATE \hspace{\algorithmicindent} \hspace{\algorithmicindent} $B_{21} \leftarrow -A_{22}^{-1} A_{21} S^{-1}$
\STATE \hspace{\algorithmicindent} \hspace{\algorithmicindent} $B_{22} \leftarrow A_{22}^{-1} + A_{22}^{-1} A_{21} S^{-1} A_{12} A_{22}^{-1}$
\STATE \hspace{\algorithmicindent} \hspace{\algorithmicindent} \textbf{return} $\begin{bmatrix} B_{11} & B_{12} \\ B_{21} & B_{22} \end{bmatrix}$
\STATE \hspace{\algorithmicindent} \textbf{end if}
\end{algorithmic}
\end{algorithm}

We examine the performance of GD and Newton-Schulz 
when accelerated by the Schur complement technique 
presented in Algorithm~\ref{alg:recursive_schur_inversion}. 
The Schur complement acceleration serves a dual purpose: 
partitioning matrices into smaller blocks 
reduces the computational burden of individual iterations, 
while the resulting submatrices typically exhibit better conditioning, 
thereby reducing the iteration count required for convergence 
in the base case solver.

We analyze both algorithms when accelerated by the Schur complement technique, 
applied to matrices with condition number $\kappa = 100$. 
The acceleration method recursively partitions matrices 
by splitting them in half along both dimensions, 
computing the inverse through the Schur complement formula 
until reaching base cases of specified minimum dimensions. 
The experimental design considers minimum Schur complement dimensions 
of $\tau = \{4, 8, 16, 32\}$ across matrix sizes from 64 to 2048, 
allowing us to examine how the choice of recursion cutoff 
affects overall performance.

For gradient descent operations, 
we pre-compute the theoretical optimal learning rates 
$\Delta t^* = 1/(\lambda_{\max} + \lambda_{\min})$ 
using eigenvalue analysis before timing begins. 
While this eigenvalue computation is not practical in real applications 
where iterative methods are preferred precisely to avoid such calculations, 
we employ this approach to ensure that gradient descent operates under optimal conditions 
and to eliminate learning rate selection 
as a confounding variable in the performance comparison.

Table~\ref{tab:block_results} presents the comparative results 
across all test configurations. 
The data reveal that gradient descent becomes increasingly competitive 
when accelerated by the Schur complement technique, 
benefiting from the improved conditioning of smaller subproblems. 
While Newton-Schulz maintains superior performance 
for the smallest matrices ($64 \times 64$), 
gradient descent achieves consistent speedups for larger matrix configurations, 
ranging from modest improvements (1.01--1.18) for the largest matrices 
to more substantial gains (up to 1.44) for intermediate sizes. 
Notably, gradient descent outperforms Newton-Schulz 
in all configurations for ($512 \times 512$) and ($1024 \times 1024$) matrices, 
indicating that the Schur complement acceleration 
appears to mitigate gradient descent's sensitivity to poor conditioning 
by operating on better-conditioned subproblems.

\begin{table*}
  \caption{\bl{Schur Complement Approximate Matrix Inversion Performance Comparison.
       We evaluate Gradient Descent (GD) versus Schulz-Hotelling (SH) iteration for approximately inverting block-structured matrices using the Schur complement acceleration.
       All matrices have condition number 100, with dimensions ranging from 64 to 2048 and varying minimum Schur complement block sizes.
       As matrix dimension increases, the performance gap narrows, and both algorithms remain competititve.
       Speedup represents the ratio of GD execution time to SH execution time.}}
  \label{tab:block_results}
  \resizebox{\textwidth}{!}{%
  \begin{tabular}{ccccccccc}
  \toprule
   \textbf{Dim} &  \begin{tabular}{@{}c@{}}\textbf{Min Schur} \\ \textbf{Block Size}\end{tabular} &  \begin{tabular}{@{}c@{}}\textbf{Iterations} \\ \textbf{GD}\end{tabular} &  \begin{tabular}{@{}c@{}}\textbf{Iterations} \\ \textbf{SH}\end{tabular} &  \begin{tabular}{@{}c@{}}\textbf{Time} \\ \textbf{GD (s)}\end{tabular} &  \begin{tabular}{@{}c@{}}\textbf{Time} \\ \textbf{SH (s)}\end{tabular} &  \begin{tabular}{@{}c@{}}\textbf{Error} \\ \textbf{GD}\end{tabular} &  \begin{tabular}{@{}c@{}}\textbf{Error} \\ \textbf{SH}\end{tabular} &  \begin{tabular}{@{}c@{}}\textbf{Speedup} \\ \textbf{($\times$)}\end{tabular} \\
  \midrule

   64 & 4 & 192 & 103 & 0.0013 & 0.0012 & 1.38e-02 & 3.27e-03 & 0.92 \\
   64 & 8 & 161 & 68 & 0.0007 & 0.0006 & 1.06e-02 & 7.87e-03 & 0.86 \\
   64 & 16 & 188 & 44 & 0.0009 & 0.0004 & 9.98e-03 & 6.56e-03 & 0.44 \\
   64 & 32 & 285 & 29 & 0.0016 & 0.0003 & 5.48e-03 & 1.64e-03 & 0.19 \\
  \midrule

   128 & 4 & 256 & 190 & 0.0015 & 0.0017 & 1.65e-02 & 1.29e-02 & 1.13 \\
   128 & 8 & 167 & 120 & 0.0011 & 0.0011 & 1.20e-02 & 2.88e-03 & 1.00 \\
   128 & 16 & 139 & 76 & 0.0008 & 0.0009 & 9.35e-03 & 1.62e-03 & 1.13 \\
   128 & 32 & 194 & 48 & 0.0014 & 0.0006 & 6.71e-03 & 5.03e-03 & 0.43 \\
  \midrule

   512 & 4 & 817 & 681 & 0.0076 & 0.0083 & 2.81e-02 & 2.37e-02 & 1.09 \\
   512 & 8 & 486 & 450 & 0.0050 & 0.0072 & 1.97e-02 & 2.67e-03 & 1.44 \\
   512 & 16 & 298 & 261 & 0.0041 & 0.0050 & 1.54e-02 & 7.13e-03 & 1.22 \\
   512 & 32 & 203 & 156 & 0.0038 & 0.0042 & 1.15e-02 & 4.67e-03 & 1.11 \\
  \midrule

   1024 & 4 & 1565 & 1304 & 0.0251 & 0.0274 & 3.52e-02 & 2.98e-02 & 1.09 \\
   1024 & 8 & 873 & 887 & 0.0267 & 0.0311 & 2.81e-02 & 1.49e-02 & 1.16 \\
   1024 & 16 & 525 & 517 & 0.0186 & 0.0220 & 1.91e-02 & 3.12e-03 & 1.18 \\
   1024 & 32 & 309 & 299 & 0.0172 & 0.0199 & 1.49e-02 & 4.82e-03 & 1.16 \\
  \midrule

   2048 & 4 & 2907 & 2577 & 0.1279 & 0.1287 & 4.78e-02 & 3.18e-02 & 1.01 \\
   2048 & 8 & 1601 & 1677 & 0.1180 & 0.1212 & 3.52e-02 & 4.29e-02 & 1.03 \\
   2048 & 16 & 916 & 1024 & 0.1186 & 0.1259 & 2.55e-02 & 3.24e-03 & 1.06 \\
   2048 & 32 & 524 & 581 & 0.1514 & 0.1253 & 2.02e-02 & 4.15e-03 & 0.83 \\

  \bottomrule
  \end{tabular}}
\end{table*}

\begin{remark}
The performance results presented use pre-computed optimal learning rates 
to demonstrate the algorithm's capabilities under ideal conditions. 
In practical applications, adaptive learning rate strategies would be employed, 
potentially achieving 70--90\% of this optimal performance 
without eigenvalue computation.
\end{remark}

{Finally, we validate the theoretical optimal learning rate 
through sensitivity analysis.}
We systematically evaluated learning rates 
within the range $[0.1\Delta t^*, 2.0\Delta t^*]$ 
using 50 equally spaced sampling points. 
Figure 7 presents the results 
for matrices with varying condition numbers, 
illustrating both the number of iterations required for convergence 
(convergence criterion: $\epsilon = 10^{-6}$) as a function of the learning rate 
(left panel) 
and the corresponding spectral radius as a function of the learning rate 
(right panel). 
These experimental results corroborate the theoretical predictions 
established in Theorems~\ref{thm:convergence_condition} and~\ref{thm:optimal_step_size}.

\begin{figure}[!t]
\centering
\includegraphics[width=1.0\textwidth]{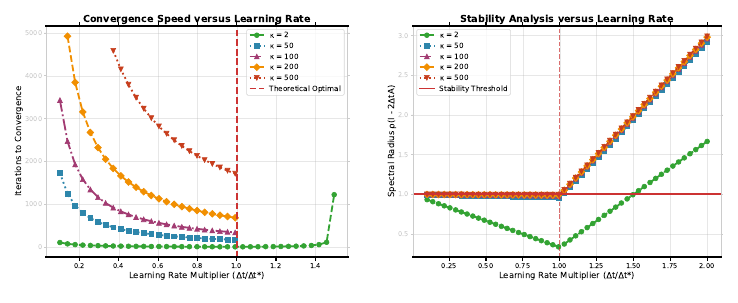}
\label{fig:spectral}
\end{figure}

\section*{Discussion}

{The first-order equivalence between thermodynamic SPD matrix inversion 
and preconditioned gradient descent 
situates our results within the broader landscape of analog computing, 
where physical processes have long served as computational substrates.}
Analog computing represents a computational paradigm 
where physical processes directly embody 
the mathematical operations being performed.
The original analog computers were mechanical systems 
that solved differential equations 
by using the characteristics of physical processes, 
such as spring mechanics, fluid dynamics, or electrical circuits, 
to create \textit{analogues} of other dynamical systems~\cite{owens1986vannevar}.
In this approach, the time evolution of the physical system directly mimics 
the behavior described by the differential equations 
that would otherwise require numerical solution.

The connection between analog computing and optimization 
extends naturally to oscillatory neural networks,
where coupled oscillators evolve according to dynamical equations 
that minimize energy functions through
their collective behavior~\cite{todri_sanial_computing_2024}. 
These systems exploit the natural tendency
of physical oscillators to synchronize and settle into stable configurations, 
effectively performing optimization through their intrinsic dynamics. 
Thermodynamic computing, as demonstrated by Aifer et al.~\cite{aifer_thermodynamic_2024},
represents a specific instance of this broader paradigm 
where thermal fluctuations drive the system 
toward statistical equilibrium states that encode computational solutions.

Our analysis reveals a direct mathematical connection 
between these analog processes and classical optimization algorithms. 
Consider the differential equation that
governs the evolution of our approximate matrix inverse:
\begin{equation*}
    \frac{\mathrm{d}\tilde{A}^{-1}}{\mathrm{d}t} = -2(\tilde{A}^{-1}A - I).
\end{equation*}

This ordinary differential equation represents the deterministic core 
of the thermodynamic matrix inversion process, 
stripped of its stochastic thermal components.
When discretized using forward Euler integration, 
these continuous-time dynamics yield:
\begin{alignat}{2}
                    & &  \frac{\tilde{A}^{-1}_{i+1} - \tilde{A}^{-1}_i}{\Delta t} & = -2(\tilde{A}^{-1}_iA - I) \\
   \Rightarrow\quad & & \tilde{A}^{-1}_{i+1} & = \tilde{A}^{-1}_i -2\Delta t(\tilde{A}^{-1}_iA - I)+ O\left(\Delta t^2 \right),
\end{alignat}
which is precisely the gradient descent rule derived in Theorem~\ref{theo:ou_gradient_descent} to first-order in $\Delta t$ when quadratic terms are neglected.

This connection illuminates the relationship 
between analog and digital implementations of the same mathematical process. 
While Aifer et al.~\cite{aifer_thermodynamic_2024}
demonstrated SPD matrix inversion through the long-term statistical behavior 
of the stochastic Ornstein-Uhlenbeck process, 
our analysis shows that the essential computational content can be captured 
by a much simpler deterministic ordinary differential equation. 
The thermal noise that drives equilibration in the physical system 
corresponds to the iterative updates that drive
convergence in the optimization algorithm.

For the specific case of thermodynamic matrix inversion, 
the analysis presented in the {\nameref{sec:gradient_descent}} section suggests
that the sophisticated stochastic machinery may not be computationally essential; 
the same mathematical result can be achieved
through straightforward iterative optimization. 
Whether similar optimization equivalences exist for other analog computing
approaches remains an open question that could benefit from theoretical investigation.
The connection we establish provides one example 
of how physical computing paradigms can be understood through
classical optimization theory, 
potentially offering insights for both theoretical analysis 
and practical implementation of analog computing systems.

This mathematical equivalence has concrete implications 
for the design and understanding of thermodynamic computing systems.
The most striking implication of Theorem~\ref{theo:ou_gradient_descent} 
is that the thermal fluctuations central to thermodynamic computing 
are not algorithmically essential for SPD matrix inversion. 
The complex stochastic dynamics of the Ornstein-Uhlenbeck process 
reduce to a deterministic iteration that captures the same mathematical content.

We note, however, that this redundancy is specific to problems 
that involve a convex quadratic potential with a single global minimum, 
of which matrix inversion is one. 
In broader thermodynamic computing applications 
involving non-convex energy landscapes 
(e.g., combinatorial optimization or sampling from multi-modal distributions), 
thermal fluctuations could remain algorithmically essential 
to escape local minima and cross energy barriers.
Our result clarifies that for the subclass of convex problems such as linear algebra, 
the \textit{annealing} property of the noise is not required,
allowing the process to be efficiently collapsed into deterministic moment evolution.

Despite this approximate algorithmic equivalence, 
physical thermodynamic systems may still offer advantages in specific contexts.
For specific hardware constraints and problem sizes, 
analog thermodynamic circuits may achieve better energy efficiency 
than digital processors performing equivalent gradient descent iterations. 
Additionally, when additional physical constraints 
are naturally incorporated into the thermodynamic system, 
such as positivity constraints through circuit design, 
the physical implementation may handle constrained matrix inversion more naturally 
than digital methods that require explicit constraint handling.

These observations provide practical guidance for hardware design.
The time constant of the thermodynamic system 
should be chosen to match the optimal learning rate 
$\Delta t^* = 1/(\lambda_{\max} + \lambda_{\min})$ for faster convergence.
The physical implementation should ensure 
that the effective preconditioner matches the target matrix $A$, 
suggesting design principles for coupling strengths and thermal reservoir properties.
Since the physical system implements gradient descent, 
standard optimization convergence criteria can be applied 
to determine when thermal equilibration is sufficient, 
replacing ad-hoc equilibration time estimates with principled convergence monitoring.

In summary, this work has established a theoretical connection 
between thermodynamic computing and classical optimization 
by demonstrating that, to first-order in the discretized dynamics, thermodynamic matrix inversion on SPD matrices
is mathematically equivalent to preconditioned gradient descent. 
This equivalence reveals that the moment evolution 
of the complex stochastic dynamics underlying thermodynamic computing 
can be understood through the framework of deterministic optimization. 
Consequently, the stochastic thermal fluctuations 
that are central to the analog approach 
are shown to be algorithmically inessential 
for the specific convex problem of SPD matrix inversion, 
though they remain fundamental 
for thermodynamic applications involving non-convex energy landscapes.

The theoretical connection we establish 
contributes to the broader understanding 
of how physical processes relate to computational algorithms. 
By revealing the optimization structure underlying thermodynamic computing, 
our work demonstrates how theoretical analysis can illuminate 
the mathematical foundations of unconventional computing approaches. 
This finding has significant implications for both computational paradigms: 
it provides theoretical grounding for thermodynamic computing approaches 
while simultaneously offering new perspectives 
on how physics-inspired algorithms relate 
to standard numerical linear algebra methods.

%
%
\section*{\bl{Methods}}
\subsection*{Mathematical Preliminaries}\label{sec:math_prelims}

This section presents the mathematical tools necessary for our analysis. 
This includes important concepts from linear algebra, stochastic processes, and discrete probability 
representations that are needed to support our main theoretical results.

\paragraph{}

We begin with important matrix operations that are central to our analysis. 
Given two matrices $A\in \mathbb{R}^{n\times m}$ and $B \in \mathbb{R}^{p\times q}$, we let $A\otimes B$ denote the $pn\times qm$ Kronecker product defined by 
\begin{equation}    
A\otimes B = \begin{bmatrix}
  a_{11} B & \cdots & a_{1m} B \\
             \vdots & \ddots &           \vdots \\
  a_{n1} B & \cdots & a_{nm} B
\end{bmatrix}.
\end{equation}

For a matrix $A\in \mathbb{R}^{n\times m}$, we let $A^T\in \mathbb{R}^{m\times n}$ denote its transpose. When $A \in \mathbb{R}^{n\times n}$ is square, 
we define the matrix exponential by 
\begin{equation}
e^A = \sum_{k=0}^{\infty} \frac{A^k}{k!},
\end{equation}
which satisfies $\frac{d}{dt} e^{tA} = A e^{tA}$ for any scalar $t$.

\begin{definition}[Schur Complement]
\label{schurDefinition}
Let $M = \begin{pmatrix} A & B \\ C & D \end{pmatrix}$ be a block 
matrix where $A \in \mathbb{R}^{n \times n}$ and $D \in \mathbb{R}^{m \times m}$. If $A$ is invertible, 
the \textbf{Schur complement} of $A$ in $M$ is defined as
\begin{equation}
S = D - CA^{-1}B.
\end{equation}
When both $A$ and $S$ are invertible, the inverse of $M$ is given by
\begin{equation}
M^{-1} = \begin{pmatrix} 
A^{-1} + A^{-1}BS^{-1}CA^{-1} & -A^{-1}BS^{-1} \\ 
-S^{-1}CA^{-1} & S^{-1}
\end{pmatrix}.
\end{equation}
\end{definition}

\paragraph{}

Our connection between thermodynamic systems and optimization algorithms relies on the mathematical properties 
of stochastic processes, particularly the Ornstein-Uhlenbeck process \cite{uhlenbeck_theory_1930}.

\begin{definition}[Ornstein-Uhlenbeck Process]
Let $A, D \in \mathbb{R}^{n \times n}$ and $W_t$ be a standard Wiener process for $t \in [0, \infty)$ on a filtered probability space. 
The Ornstein-Uhlenbeck process $X_t \in \mathbb{R}^n$ is defined by the stochastic differential equation:
\begin{equation}
    dX_t = -A X_t dt + D \, dW_t.
\end{equation}
\end{definition}

\begin{definition}[Discretized Ornstein-Uhlenbeck Process]
\label{def:discrete_ou}
{For the specific case of thermodynamic matrix inversion, the system is designed such that the diffusion 
matrix is isotropic, and hence $D = \sqrt{2}I_n$~\cite{aifer_thermodynamic_2024}.} The Euler-Maruyama discretization~\cite{kloeden_numerical_1992} of this 
process with time step $\Delta t$ is given by:
\begin{equation}
\label{e.ou_discret}
x_{i+1} = (I - \Delta t A) x_i + \sqrt{2\Delta t} Z_i,
\end{equation}
where $x_i \in \mathbb{R}^n$ approximates $X_{i\Delta t}$, and $Z_i \sim \mathcal{N}(0, I_n)$ are independent Gaussian random vectors.
\end{definition}

\begin{definition}[Autocorrelation Matrix]
\label{def:autocorr}
For the continuous Ornstein-Uhlenbeck process $X_t$ and its discrete approximation $x_i$, the autocorrelation matrices are defined as:
\begin{equation}
\label{e.ou_sde}
\Sigma_t = \mathbb{E}[X_t \otimes X_t] = \mathbb{E}[X_t X_t^\top] \quad \text{(continuous)},
\end{equation}
\begin{equation}
\Sigma_i = \mathbb{E}[x_i \otimes x_i] = \mathbb{E}[x_i x_i^\top] \quad \text{(discrete)}.
\end{equation}
We use the bracket notation $\langle \cdot \rangle$ as shorthand for the expectation operator $\mathbb{E}[\cdot]$.
\end{definition}

\begin{proposition}[Properties of the Ornstein-Uhlenbeck Process]
\label{prop:ou_properties}
The Ornstein-Uhlenbeck process satisfies:
\begin{itemize}
    \item Explicit solution: 
    \begin{equation}
    X_t = e^{-At}X_0 + \int_0^t e^{-A(t-s)}D \, dW_s.
    \end{equation}
    
    \item Mean evolution: 
    \begin{equation}
        \mathbb{E}[X_t] = e^{-At} \mathbb{E}[X_0].
    \end{equation}
    
    \item Autocorrelation matrix (when $\mathbb{E}[X_0] = 0$): 
    \begin{equation}\label{e.autocorrelation_matrix}
        \Sigma_t = \langle X_t \otimes X_t \rangle = \int_0^t e^{-A(t-s)}DD^\top e^{-A^\top(t-s)} \, ds.
    \end{equation}
\end{itemize}
\end{proposition}

{The \nameref{sec:ou_process} section} uses these properties to establish the connection between the steady-state behavior of thermodynamic systems 
and approximate matrix inversion.

\paragraph{}
\label{sec:discrete_reps}
{
In order to digitally simulate the thermal noise inherent to thermodynamic systems, an efficient, 
discrete representation of continuous probability distributions is required. 
A broad class of such approximations represents a distribution as a weighted sum of Dirac deltas \cite{tsoutsouras_laplace_2022}:
}
\begin{equation}
\mu_{N} = \sum_{i=1}^{N} p_i \delta_{x_i},
\end{equation}

{
where $N$ is the representation size, $p_i \geq 0$ are probability masses 
with $\sum_{i=1}^N p_i = 1$, and $x_i \in \mathbb{R}$ are the support points. 
The positions and weights of such a representation depend on the quantization 
algorithm applied to the continuous distribution.
}

{A discrete representation of the noise term in Equation~\eqref{e.ou_discret} then allows for the simulation of the Ornstein-Uhlenbeck process via performing arithmetic operations directly on the probability distributions, allowing one to efficiently compute the autocorrelation matrix required for matrix inversion. A more detailed explanation on how the distributional arithmetic is performed can be found in} {the \nameref{sec:digital_implementation_via_distributional_arithmetic} section}.

{One example of such a discretization scheme is the Telescopic Torques Representation (TTR) \cite{tsoutsouras_laplace_2022}, a recursive divide-and-conquer method for constructing discrete approximations of continuous probability distributions~\cite{tsoutsouras_laplace_2022}. The algorithm is based on the principle that the mean value of a distribution provides an optimal splitting point that balances the ``probability torque''---the product of probability mass and distance from the mean---on either side of the partition.}

\begin{definition}[TTR Algorithm]
Let $T: \mathcal{M} \times \mathbb{Z}_{\geq 0} \rightarrow \mathcal{M}$ be the TTR algorithm that takes a probability measure $\mu$ and parameter $n \geq 0$ as input, returning a discrete measure with $2^n$ atoms. The algorithm proceeds recursively:

\begin{enumerate}
    \item \textbf{Base case}: If $n = 0$, return $T(\mu, 0) = \delta_{\bar{\mu}}$ where $\bar{\mu} = \int x \, d\mu(x)$ is the mean of $\mu$.
    
    \item \textbf{Recursive case}: If $n \geq 1$, partition the support using the mean:
    \begin{itemize}
        \item Define $\Omega_- = \{x : x < \bar{\mu}\}$ and $\Omega_+ = \{x : x \geq \bar{\mu}\}$
        \item Construct conditional measures $\mu_\pm = \mathbf{1}_{\Omega_\pm} \mu / \mu(\Omega_\pm)$
        \item Return: $T(\mu, n) = \mu(\Omega_-) T(\mu_-, n-1) + \mu(\Omega_+) T(\mu_+, n-1)$
    \end{itemize}
\end{enumerate}
\end{definition}

This recursive splitting process creates a binary tree structure where each node corresponds to a partition based on the conditional mean. The TTR method achieves near-optimal convergence rates for many distributions when measured using the Wasserstein-1 distance~\cite{bilgin_quantization_2025}, while maintaining computational efficiency and numerical stability under distributional arithmetic operations.

\begin{figure}[!t]
\centering
\includegraphics[width=1.0\textwidth]{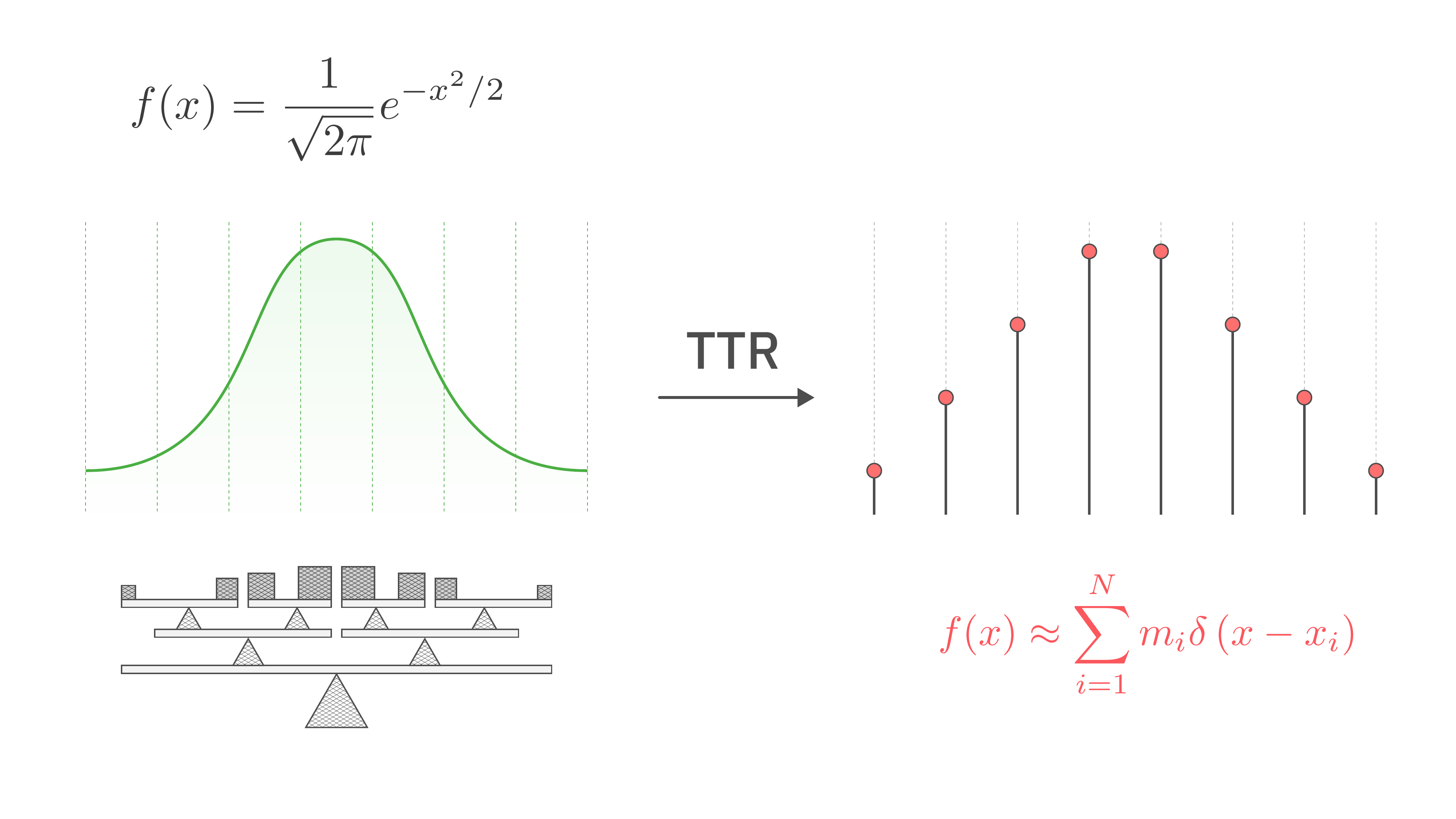}
\label{fig:ttr}
\end{figure}

The key advantage of TTR over other discrete representation methods is its preservation of important statistical properties (particularly the mean) throughout the recursive construction. This property makes TTR particularly well-suited for applications involving sequential arithmetic operations on uncertain quantities, which is precisely the setting we encounter when implementing thermodynamic algorithms digitally.

\subsection*{Approximate SPD Matrix Inversion via Ornstein-Uhlenbeck Dynamics}\label{sec:ou_process}
The connection between thermodynamic processes and linear algebra represents an interesting insight in the emerging field of thermodynamic
computing~\cite{aifer_thermodynamic_2024, coles_thermodynamic_2023}. Unlike quantum approaches to linear algebra acceleration, which face
significant hardware challenges~\cite{harrow_quantum_2009}, thermodynamic methods exploit naturally occurring thermal fluctuations in classical
systems to perform useful computation. This approach treats stochasticity as a computational resource rather than a source of error,
fundamentally reversing the traditional digital computing paradigm.

This section presents a derivation of a known result connecting approximate matrix inversion to the steady-state properties of
Ornstein-Uhlenbeck processes~\cite{uhlenbeck_theory_1930, gardiner_handbook_1985, aifer_thermodynamic_2024}.
This stochastic formulation provides the mathematical basis for establishing the bridge between analog
thermodynamic systems and digital optimization methods.

\begin{theorem}[SPD Matrix Inversion via Ornstein-Uhlenbeck Process Covariance]
\label{theorem:ou_ami}
Let $A \in \mathbb{R}^{n \times n}$ be a symmetric positive definite matrix, and let $D = \sqrt{2}I_n$.
Consider the Ornstein-Uhlenbeck process with initial condition $X_0 = 0$. Then the \bl{covariance} matrix of this process converges to the inverse of $A$ in the long-time limit:
\begin{equation}
\lim_{t \to \infty} \Sigma_t = A^{-1},
\end{equation}
where $\Sigma_t$ is the \bl{covariance} matrix given in \eqref{e.autocorrelation_matrix}.
\end{theorem}
\begin{proof}
By recalling the \bl{covariance} matrix \eqref{e.autocorrelation_matrix} (which holds for $X_0 = 0$), we see that when $A$ is symmetric and $D = \sqrt{2}I_n$ then the expression simplifies to
\begin{equation}
    \Sigma_t = 2 \int_0^t e^{-A(t-s)}e^{-A^\top (t-s)} ds =  2 \int_0^t e^{-2A(t-s)} ds.
\end{equation}
By performing the change of variables $u = t-s$ we get 
\begin{align}
\Sigma_t &= 2 \int_t^0 e^{-2Au} (-du) = 2 \int_0^t e^{-2Au} du = 2 \left[ -\frac{1}{2}A^{-1}  e^{-2Au} \right]_0^t \\
&= -A^{-1} \left( e^{-2At} - I_n \right) = A^{-1} \left( I_n - e^{-2At} \right).
\end{align}
Since $A$ is positive definite, all eigenvalues of $A$ are strictly positive, which implies that $$\lim_{t \to \infty} e^{-2At} = 0.$$ Therefore:
\begin{equation}
\lim_{t \to \infty} \Sigma_t = A^{-1} \left( I_n - 0 \right) = A^{-1}.
\end{equation}
\end{proof}

\begin{corollary}[Numerical Matrix Inversion Algorithm]
The inverse of a symmetric positive definite matrix $A$ can be approximated by solving the Ornstein-Uhlenbeck SDE numerically and computing the sample \bl{covariance} matrix of the solution for sufficiently large time $t$.
\end{corollary}

\begin{remark}
This theorem provides a probabilistic method for SPD matrix inversion. The convergence rate depends on the smallest eigenvalue of $A$, as this determines how quickly $e^{-2At}$ approaches zero.
\end{remark}

The Ornstein-Uhlenbeck formulation corresponds directly to the dynamics of thermodynamic systems, where such processes arise from thermal fluctuations.

\subsection*{Analog Implementation via Thermodynamic Computing}

In thermodynamic computing~\cite{aifer_thermodynamic_2024}, a physical system is set up so that the statistics of the distribution of states at equilibrium are a useful quantity. For example, by choosing a potential energy with a quadratic form, the mean of the equilibrium distribution is the solution to a linear system of equations.

That is, consider the system that has a potential energy given by
\begin{equation}
\label{eq:quadratic_form}
    U(x) = \frac{1}{2}x^\mathrm{T}Ax - b^\mathrm{T}x,
\end{equation}
where $x\in \mathbb{R}^n$ is the state of the system, $A \in \mathbb{R}^{n \times n}$ is the symmetric positive definite matrix and $b \in \mathbb{R}^n$ is a vector. The steady-state distribution of a system described by the quadratic form in Equation~\eqref{eq:quadratic_form} can be derived to be a normal distribution given by
\begin{equation}
\label{eq:equilibrium_distribution}
    \mathcal{N}\left(A^{-1}b, \frac{1}{\beta}A^{-1}\right),
\end{equation}
where $\beta$ is the inverse temperature once thermal equilibrium is reached. Equation~\eqref{eq:equilibrium_distribution} shows that the mean of this steady-state normal distribution is the solution to the system of linear equations $Ax = b$. By setting $b=0$, we can simplify Equation~\eqref{eq:quadratic_form} and see that the \bl{covariance} of the steady-state distribution is the matrix inverse, up to a multiplicative constant. This \bl{covariance} matrix, and therefore the matrix inverse, can be found in an analog manner by carrying on the following integration using analog multipliers and integrators,
\begin{equation}
\label{eq:covariance_integration_thermodynamic}
    x_{i}x_{j}=\frac{1}{\tau}\int_{t_{o}}^{\tau}x_{i}(t)x_{j}(t)dt,
\end{equation}
where $\tau$ is the integration time and $t_{0}$ is the time given to the system to evolve under its dynamics; these values are chosen such that a desired accuracy is reached. We note that Equation~\eqref{eq:covariance_integration_thermodynamic} computes the entry at the ($i, j$)-th position of the matrix.
{Reading the results out of the analog circuit requires $nk$ measurements, where $k$ is the number of samples per output needed to estimate the covariance matrix to the desired accuracy, and $n$ is the number of outputs (each producing a length-$k$ array of samples). Due to the ergodicity of the Ornstein-Uhlenbeck process, a single physical instantiation of the system with $n$ degrees of freedom is sufficient to compute all $n^2$ elements of the inverse matrix by integrating the pairwise products of the state variables over a single trajectory~\cite{aifer_thermodynamic_2024}.}

In order to simulate the method of calculating an SPD matrix inverse, Aifer \et~\cite{aifer_thermodynamic_2024} noted that the Langevin dynamics of their chosen system resolves to \bl{an} Ornstein-Uhlenbeck process and the \bl{covariance} of its states after a sufficiently long time period provides us the matrix inverse; see Theorem~\ref{theorem:ou_ami}.

\subsection*{Known Limitations of Thermodynamic Approaches}

Despite their theoretical advantages (speed and energy \cite{bartosik2024thermodynamicalgorithmsquadraticprogramming, Melanson2025}) thermodynamic methods for linear algebra have inherent
limitations that prevent them inverting matrices that are not symmetric positive definite \cite{aifer_thermodynamic_2024}.
However, applications using thermodynamic methods can transform general matrices by adjusting
the normal equations as $A^T A x = A^T b$ at the cost of squaring the condition
number.

In the following {\nameref{sec:digital_implementation_via_distributional_arithmetic} section}, we introduce
a method of simulating the Ornstein-Uhlenbeck process using deterministic arithmetic carried out on
distributional representations on a processor that can perform arithmetic on digital representations of probability distributions.
This approach addresses many of the limitations and challenges inherent in analog thermodynamic
implementations. {It
provides insight into the question of whether an external source of stochasticity is necessary in all
problems of analog thermodynamic systems it mimics. These insights lead to the further foundational insights
we present later in the} {\nameref{sec:gradient_descent} section.}

\subsection*{Digital Implementation Approaches: {Langevin Dynamics} Simulation}

An alternative to both traditional {simulation of the Langevin dynamics} and analog thermodynamic
hardware involves implementing the Ornstein-Uhlenbeck process through computational approaches that avoid the engineering
challenges of physical systems while maintaining mathematical fidelity.

\begin{figure}[!t]
\centering
\includegraphics[clip, trim={3cm 0cm 0cm 6.7cm}, width=1.08\textwidth]{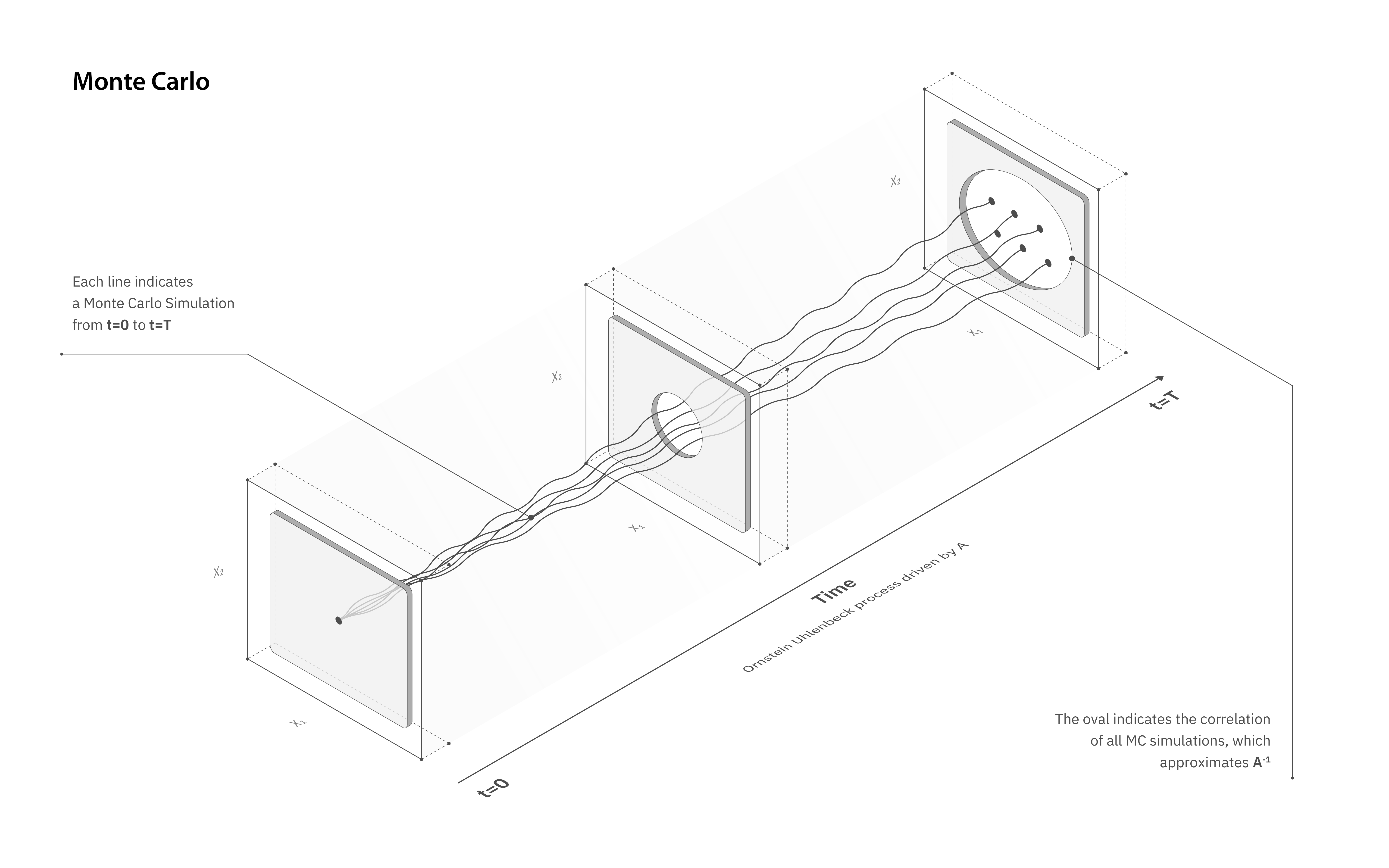}
\label{fig:mc_ou}
\end{figure}

The standard {Langevin dynamics} simulation implementation discretizes the Ornstein-Uhlenbeck SDE
using the Euler-Maruyama method~\cite{kloeden_numerical_1992}. At each time step, \bl{the algorithm draws random variates} to
represent thermal noise, generating an ensemble of trajectories that provide statistical characterization of the
process behavior (Algorithm~\ref{alg:mc_euler_maruyama_ou} \bl{and} Figure 3). This approach generates $M$
samples at each of the $N$ timesteps, with estimation error decreasing proportionally to $1/\sqrt{M}$.

\begin{algorithm}[htbp]
\caption{{Langevin Dynamics} Euler-Maruyama Method}
\label{alg:mc_euler_maruyama_ou}
\begin{algorithmic}[1]
\REQUIRE Initial condition $X_0 \in \mathbb{R}^n$, matrices $A, D \in \mathbb{R}^{n \times n}$, time step $\Delta t > 0$, final time $T$, number of {independent trajectories} $M$
\ENSURE Sample paths $\{X^{(m)}_k\}_{k=0,m=1}^{N,M}$ where $N = T/\Delta t$
\FOR{$m = 1$ to $M$}
    \STATE Initialize $X^{(m)}_0 = X_0$ and set $k = 0$
    \WHILE{$k \Delta t < T$}
        \STATE Generate $Z^{(m)}_k \in \mathbb{R}^n$ with $Z^{(m)}_k \sim \mathcal{N}(0, I_n)$
        \STATE Compute drift: $d^{(m)}_k = -AX^{(m)}_k$
        \STATE Compute diffusion: $s^{(m)}_k = D\sqrt{\Delta t} Z^{(m)}_k$
        \STATE Update: $X^{(m)}_{k+1} = X^{(m)}_k + d^{(m)}_k \Delta t + s^{(m)}_k$
        \STATE $k \leftarrow k + 1$
    \ENDWHILE
\ENDFOR
\RETURN $\{X^{(m)}_k\}_{k=0,m=1}^{N,M}$
\end{algorithmic}
\end{algorithm}

\subsection*{{Digital Implementation Approaches: Distributional Arithmetic}}\label{sec:digital_implementation_via_distributional_arithmetic}

A more \bl{efficient} approach involves direct arithmetic on discretized probability
distributions, such as the Telescoping Torques Representation (TTR) discussed in
{the \nameref{sec:math_prelims} section}. Rather than generating random samples, this method performs \bl{arithmetic on representations
of distributions} directly, where each operation between uncorrelated random
variables is implemented as a convolution of their discrete representations~\cite{tsoutsouras_laplace_2022}.

This approach enables implementation of the Euler-Maruyama discretization of the Ornstein-Uhlenbeck process
without {simulating several independent trajectories}. Instead of generating random variates
at each time step, the method performs arithmetic directly on the discrete representation of the standard
normal distribution, propagating the full probability distribution through each iteration (Algorithm~\ref{alg:euler_maruyama_ou_multivariate} and Figure 4).
This eliminates {sampling} variance while maintaining the complete statistical characterization of the process.

While \bl{arithmetic on representations of distributions}
eliminates {sampling} variance, it faces memory-related challenges in tracking
correlations between variables. For matrix inversion applications, the interdependencies between
matrix elements significantly affect computational accuracy, leading to exponential memory
requirements when correlations are tracked naively. The naive approach of tracking all possible
correlations leads to a curse of dimensionality, as the joint distribution over $n$ correlated
variables requires enumeration of all possible outcome combinations, resulting in computational
complexity that grows as $O(k^n)$ where $k$ is the discrete representation size.

However, in forthcoming work, we present a tensor train decomposition \cite{oseledets2011tensor} methodology
that efficiently represents high-dimensional correlated distributions while avoiding the exponential memory scaling.
This approach enables a practical implementation of matrix inversion and other multivariate
computations on representations of probability distributions.

\subsection*{Reduction to Simple Optimization}

The computational overhead of both {stochastic simulation} and distributional
correlation tracking motivates investigation of whether the underlying mathematical structure
can be extracted more directly. \bl{As we demonstrate in the {\nameref{sec:gradient_descent} section}, the entire computational pathway for
approximate SPD matrix inversion via the Ornstein-Uhlenbeck process can be elegantly reduced
to a simple deterministic optimization algorithm}.

This reduction reveals that the complex machinery developed for implementing thermodynamic processes digitally,
while mathematically rigorous, is not computationally essential for matrix inversion. \bl{The thermal noise that drives the physical process is algorithmically
redundant for matrix inversion, as the essential dynamics reduce to deterministic gradient descent on the covariance evolution. This insight simplifies implementation
by removing the need for stochastic simulation and provides theoretical clarity about the optimization structure underlying thermodynamic linear algebra computations.}

\begin{figure}[!t]
\centering
\includegraphics[clip, trim={2cm 1cm 0cm 0cm}, width=1.08\textwidth]{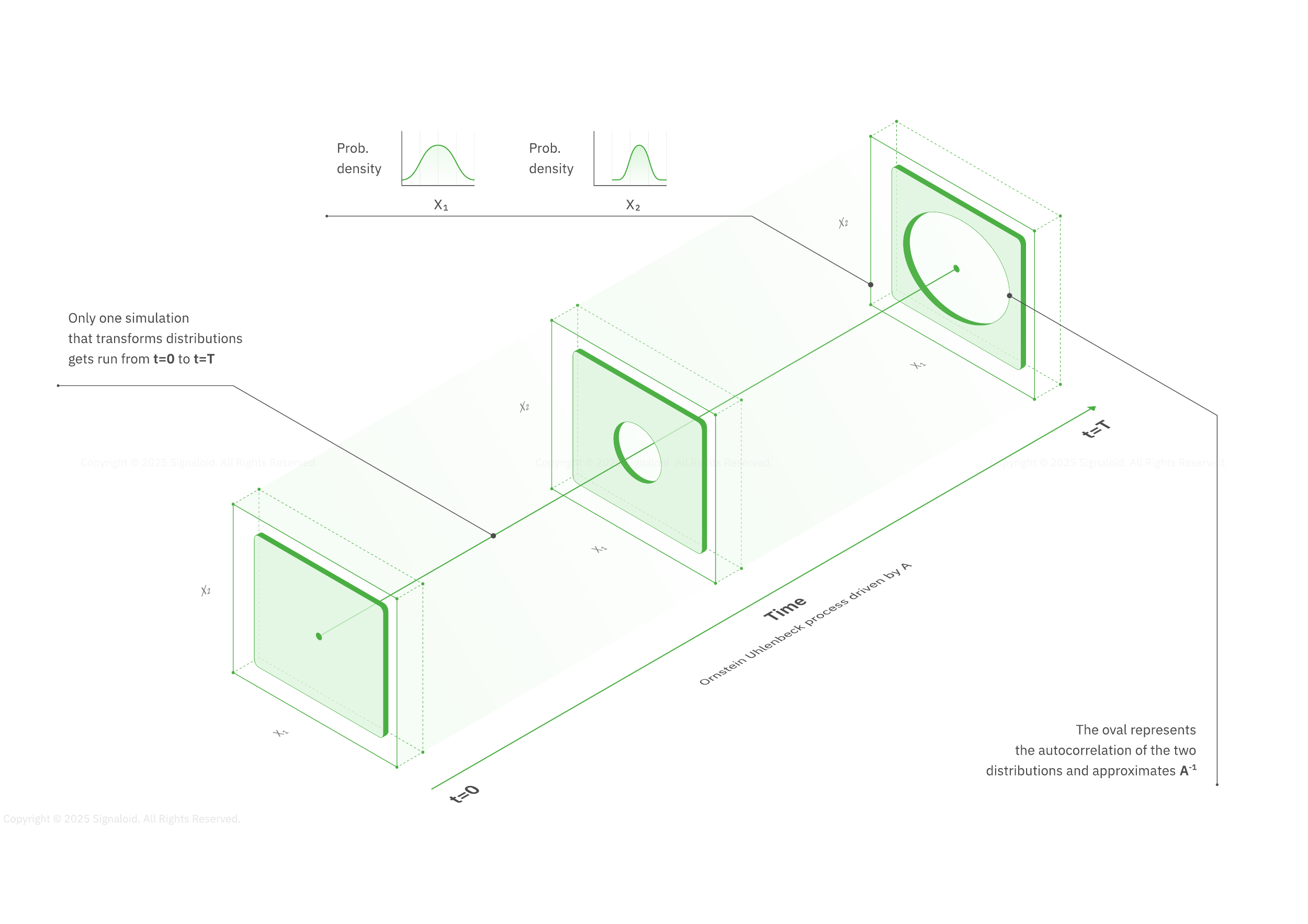}
\label{fig:signaloid_mc}
\end{figure}

\begin{algorithm}[!t]
\caption{Deterministic Distributional Euler-Maruyama Method}
\label{alg:euler_maruyama_ou_multivariate}
\begin{algorithmic}[1]
\REQUIRE Initial condition $X_0 \in \mathbb{R}^n$, matrices $A, D \in \mathbb{R}^{n \times n}$, time step $\Delta t > 0$, final time $T$
\ENSURE Approximate solution path $\{X_k\}_{k=0}^{N}$ where $N = T/\Delta t$
\STATE Initialize $X_0$ and set $k = 0$
\WHILE{$k \Delta t < T$}
    \STATE {Construct $Z_k$ as the discrete representation of $\mathcal{N}(0, I_n)$ (deterministic distribution object, not a random sample)}
    \STATE Compute drift term: $d_k = -AX_k$
    \STATE Compute diffusion term: $s_k = D\sqrt{\Delta t} Z_k$
    \STATE Update: $X_{k+1} = X_k + d_k \Delta t + s_k$ {\COMMENT{Operations denote distributional arithmetic}}
    \STATE $k \leftarrow k + 1$
\ENDWHILE
\RETURN $\{X_k\}_{k=0}^{N}$
\end{algorithmic}
\end{algorithm}

\section*{Data Availability}
{The numerical experiments in this study use synthetically generated test matrices 
constructed via spectral decomposition with predetermined condition numbers, 
as described in the Numerical Experiments section. No external datasets were analysed.}

\section*{Code Availability}
{The Python implementations used for the numerical experiments are available 
from the corresponding author upon request.}

\section*{Acknowledgements}

{This research was carried out under contract to the 
UK Advanced Research and Innovation Agency (UK ARIA) 
under contract number NACB-PR01-P02.}

\section*{Author Contributions}

{The manuscript was written by G.K., M.S., J.P., O.H.E, J.M., and P.S.M. 
The algorithms were proposed by M.S., G.K., and P.S.M. The main analytical results were derived 
by M.S. The numerical experiments were performed by 
G.K. The figures were created by 
G.K. and J.M. All authors have read and approved the manuscript.}

\section*{Competing Interests}

{The authors declare no competing financial or non-financial interests.}

\bibliography{references}

\section*{Figure Legends}

\textbf{Figure 1} $|$ A visual representation of the TTR discretization.

\textbf{Figure 2} $|$ A visual representation of how Langevin dynamics simulation is applied to approximate the inverse of an SPD matrix via the Ornstein-Uhlenbeck process.

\textbf{Figure 3} $|$ A visual representation of how random variable transformations via distributional arithmetic is applied to approximate the inverse of an SPD matrix via the Ornstein-Uhlenbeck process.

\textbf{Figure 4} $|$ Algorithmic comparison of Preconditioned Gradient Descent (left) and stochastic simulation (right) for approximate symmetric positive definite (SPD) matrix inversion. Both algorithms execute identical computational steps, indicated by matching background colours in the pseudocode: initialisation $\Sigma_0 = \alpha I$ (purple), where $\Sigma_0$ is the initial covariance estimate, $\alpha > 0$ is a scalar initialisation parameter, and $I$ is the identity matrix; iteration structure (orange), where $T$ is the total simulation time and $\Delta t$ is the time step; residual computation $R_i = \Sigma_i A - I$ (red), where $R_i$ is the residual at iteration $i$, $\Sigma_i$ is the current covariance estimate, and $A$ is the SPD matrix to be inverted; covariance update (green); and convergence criteria (pink), evaluated using the Frobenius norm $\lVert \cdot \rVert_F$ against tolerance $\epsilon$. The boxed stochastic machinery on the right (drift computation $\mathrm{drift}_i = -A x_i$ and noisy state evolution $x_{i+1} = x_i + \mathrm{drift}_i \Delta t + \sqrt{2 \Delta t}\, Z_i$, where $x_i$ is the state vector at iteration $i$ and $Z_i \sim \mathcal{N}(0, I)$ is a standard Gaussian noise term) represents the expensive computational overhead that Theorem~\ref{theo:ou_gradient_descent} proves unnecessary. The symbol $\otimes$ denotes the outer product and $\langle \cdot \rangle$ denotes the expectation.

\textbf{Figure 5} $|$ Validation of GD against the stochastic baseline (Thermox). Left: Computational time required by Thermox as a function of sample count, with the GD baseline time shown as a horizontal reference line. Right: Approximation error versus number of samples, with the theoretical Langevin dynamics simulation convergence rate of $1/\sqrt{N_s}$ shown for comparison.

\textbf{Figure 6} $|$ Sensitivity analysis of learning rate performance across different matrix condition numbers. Left panel: Number of iterations required for convergence ($\epsilon = 10^{-6}$) as a function of learning rate. Right panel: Spectral radius as a function of learning rate.

\end{document}